\documentclass[a4paper,english,oneside,10pt]{article}
\usepackage[T1]{fontenc}
\usepackage[utf8]{inputenc} 
\usepackage{authblk}

\usepackage{geometry}
\usepackage{amsmath}
\usepackage{amsfonts}
\usepackage{amssymb}
\usepackage{amsthm}
\usepackage[english]{babel}
\usepackage{color}
\usepackage{lmodern}
\usepackage{pstricks-add}
\usepackage{hyperref}
\hypersetup{hidelinks}

\usepackage{mathscinet}
\usepackage{enumitem}
\usepackage[cal=boondoxo,bb=ams]{mathalfa}
\usepackage{microtype}

\usepackage{mathtools}
\usepackage{esint}

\theoremstyle{plain}
\newtheorem{theorem}{Theorem}[section]
\newtheorem{lemma}[theorem]{Lemma}

\theoremstyle{definition}
\newtheorem{definition}[theorem]{Definition}

\theoremstyle{remark}
\newtheorem{remark}{Remark}

\begin{document}

\title{A blow-up result for the semilinear Euler-Poisson-Darboux-Tricomi equation with critical power nonlinearity}

\author[1]{Ning-An Lai}
\author[2]{Alessandro Palmieri}
\author[3]{Hiroyuki Takamura}
\affil[1]{School of Mathematical Sciences, Zhejiang Normal University, Jinhua 321004, China}
\affil[2]{Department of Mathematics, University of Bari, Via E. Orabona 4, 70125 Bari, Italy}
\affil[3]{ Mathematical Institute, Tohoku University, Aoba, Sendai 980-8578, Japan}

\renewcommand\Authands{ and }

\maketitle

\begin{abstract}
In this paper, we prove a blow-up result for a generalized semilinear Euler-Poisson-Darboux equation with polynomially growing speed of propagation, when the power of the semilinear term is a shift of the Strauss' exponent for the classical semilinear wave equation. Our proof is based on a comparison argument of Kato-type for a second-order ODE with time-dependent coefficients, an integral representation formula by Yagdjian and the Radon transform. As byproduct of our method, we derive upper bound estimates for the lifespan which coincide with the sharp one for the classical semilinear wave equation in the critical case.

\end{abstract}

\begin{flushleft}
\textbf{Keywords} critical exponent, Strauss' exponent, blow-up, power nonlinearity, Tricomi equation, Euler-Poisson-Darboux equation
\end{flushleft}

\begin{flushleft}
\textbf{AMS Classification (2020)} 
35B33, 35B44, 35C15, 35L15, 35L71, 35Q05.
\end{flushleft}

\section{Introduction}

Let us consider the following semilinear Cauchy problem with \emph{power nonlinearity} $|u|^p$
\begin{align}\label{semilinear EPDT}
\begin{cases} \partial_t^2u-t^{2\ell}\Delta u + \mu t^{-1} \partial_t u+\nu^2 t^{-2} u=| u|^p, &  x\in \mathbb{R}^n, \ t>1,\\
u(1,x)=\varepsilon u_0(x), & x\in \mathbb{R}^n, \\ \partial_t u(1,x)=\varepsilon u_1(x), & x\in \mathbb{R}^n,
\end{cases}
\end{align} where $n\geqslant 1$ denotes the space dimension, $\ell >-1$, $\mu,\nu^2$ are nonnegative real constants, $p>1$, $\varepsilon$ is a positive quantity describing the size of the data and $u_0,u_1$ are compactly supported. 

For brevity, we  call the equation in \eqref{semilinear EPDT} \emph{semilinear Euler-Poisson-Darboux-Tricomi equation} (semilinear EPDT equation) since for  $\ell=0$ and $\nu^2=0$, the linearized equation is called the \emph{Euler-Poisson-Darboux equation} \cite{Eu,Poi,Da,DW53,Co56} while for $\mu=\nu^2=0$ the second-order operator $\partial^2_t-t^{2\ell}\Delta$ is known as \emph{generalized Tricomi operator}, see \cite{Tri,Wei54,DL71}.

As usual for wave equations with time dependent coefficients $\mu t^{-1}$ for the damping term and $\nu^2 t^{-2}$ for the mass term, we introduce the quantity
\begin{align}\label{def delta}
\delta\doteq (\mu-1)^2-4\nu^2
\end{align} 
which is helpful to describe the interplay between the damping and the mass term (cf. \cite{NPR16,PalRei18,Pal18odd,Pal18even,DabbPal18,Pal19RF}). 
 We underline that, roughly speaking, the requirement $\delta\geqslant 0$ guarantees that $\mu t^{-1} \partial_t u$ has in some way a dominant influence over $\nu^2 t^{-2} u$. 

 Recently, in \cite{Pal21} it is show that for $\mu,\nu^2\geqslant 0$ 
satisfying $\delta\geqslant 0$ a blow-up result holds for suitable local solutions to \eqref{semilinear EPDT} when the power $p$ in the nonlinear term fulfills
\begin{align}\label{subcritical range p}
1<p< \max\left\{p_{\mathrm{Str}}\!\left(n+\tfrac{\mu}{\ell+1},\ell\right),p_{\mathrm{Fuj}}\!\left((\ell+1)n+\tfrac{\mu-1}{2}-\tfrac{\sqrt{(\mu-1)^2-4\nu^2}}{2}\right) \right\}.
\end{align}  Here $p_{\mathrm{Fuj}}(n)\doteq 1+\frac{2}{n}$ denotes the \emph{Fujita exponent} (named after the author of \cite{Fuj66}) and its presence can be justified as a consequence of diffusion phenomena between the solutions to the corresponding linearized model and those to some suitable parabolic equation. On the other hand, the exponent $p_{\mathrm{Str}}(n,\ell)$ is the critical exponent for the semilinear generalized-Tricomi equation established in the series of papers  \cite{HWY17,HWY16,HWY17d2,TL19,HWY20,Sun21}, which generalize the celebrated \emph{Strauss' exponent} (cf. \cite{John79,Kato80,Str81,Glas81B,Glas81,Sid84,Scha85,LinSog96,Geo97,YZ06,Zhou07}) for the classical semilinear wave equation, that is recovered for $\ell=0$. More precisely,  $p_{\mathrm{Str}}(n,\ell)$ is the larger root of the quadratic equation
\begin{equation}\label{crit exponent tricomi}
\left(\tfrac{n-1}{2}+\tfrac{\ell}{2(\ell+1)}\right)p^2-\left(\tfrac{n+1}{2}-\tfrac{3\ell}{2(\ell+1)}\right)p-1=0.
\end{equation}
Hence, the blow-up range obtained in \cite{Pal21}, described in \eqref{subcritical range p}, is expressed through shifts of Fujita exponent and of the generalized Strauss' exponents defined by \eqref{crit exponent tricomi}.

Over the last decade, several blow-up results have been proved for the Cauchy problem \eqref{semilinear EPDT} in the special case $\ell=0$; this equation is also known as wave equation with scale-invariant damping and mass terms. Furthermore, a number of different blow-up techniques have been employed to show the nonexistence of global in time solutions under suitable assumptions on $p$ and on the Cauchy data. More precisely, the following methods have been applied:
\begin{itemize}
\item the \emph{test function method} (in the formulation of Mitidieri-Pohozaev \cite{MPbook}) that exploits the scaling-properties of the partial differential operator, cf. \cite{DL13,Wak14,NPR16};
\item a \emph{modified test function method} in which the choice of the test function is deeply related to the fundamental solution of the partial differential operator, cf. \cite{IS17,PT18};
\item a \emph{Kato-type comparison principle} for second-order ODE for studying the blow-up dynamic of the space average of  a local solution, cf. \cite{DLR15,LTW17,PalRei18};
\item an \emph{iteration argument} to derive a sequence of lower bound estimates for a suitable function associated with a  local solution, cf. \cite{PT18,LST20,CGL21,GL22}.
\end{itemize} 

We underline that the choice of the most suitable technique among the ones listed above depends heavily on the range for the constants $\mu$ and $\nu^2$ in the time-dependent coefficients and their mutual interaction described through the number $\delta$ defined in \eqref{def delta}. We stress that the quantity $\delta$ is an invariant for the linearized equation under multiplication for powers of $t$: if $\varphi$ solves the linear equation $$\varphi_{tt}-t^{2\ell}\Delta \varphi+\mu t^{-1} \varphi_t +\nu^2t^{-2} \varphi=0$$ then, fixed $\vartheta \in\mathbb{R}$, the function $\psi=t^\vartheta\varphi$ solves $$\psi_{tt}-t^{2\ell}\Delta \psi+(\mu-2\vartheta) t^{-1} \varphi_t +(\vartheta^2-(\mu-1)\vartheta+\nu^2)t^{-2} \varphi=0$$ with 
\begin{align*}
\delta(\psi)& =(\mu-2\vartheta-1)^2-4(\vartheta^2-(\mu-1)\vartheta+\nu^2) \\ & =(\mu-1)^2 -4\vartheta(\mu-1)+4\vartheta^2-4(\vartheta^2-(\mu-1)\vartheta+\nu^2) = (\mu-1)^2-4\nu^2= \delta(\varphi).
\end{align*}

The aim of the present paper is to prove that the local in time solutions to \eqref{semilinear EPDT} blow-up in finite time when the power of the nonlinear terms $p$ is equal to the threshold value $p_{\mathrm{Str}}(n+\frac{\mu}{\ell+1},\ell)$, under suitable sign assumptions for the Cauchy data, and to derive upper-bound estimates for the lifespan. Other attempts have been done in the literature to study this type of critical case (cf. \cite{PalRei18,PT18,LST20,Pal20EdSmu}). In particular, we mention the critical case for $\ell\in (-1,0)$ and $\nu^2=0$ was treated in \cite{Pal20EdSmu} by using the approach for the classical wave equation introduced in \cite{WakYor19}. Unfortunately, in \cite{Pal20EdSmu} technical restrictions on the range for $\mu$ appear. 
In our result, these kinds of restrictions are not present (we only assume that $\delta$ is nonnegative and that $p=p_{\mathrm{Str}}(n+\frac{\mu}{\ell+1},\ell)$). Indeed, we use a more precise (though quite elaborate) analysis of the time evolution of the space average of a local solution, which is based on the approach from \cite{TakWak11} for the critical semilinear wave equation in higher dimensions.

\begin{remark} The blow-up in finite time of local solution in the other critical case $p=p_{\mathrm{Fuj}}\left((\ell+1)n+\tfrac{\mu-1-\sqrt{\delta}}{2}\right)$ has been investigated in \cite[Section 3]{Pal21}. In this second threshold case, it is easier to study the blow-up in finite time of a local solution, as one can exploit the diffusive nature of the model without the necessity to estimate the fundamental solution.
\end{remark}

\paragraph{Notations}  We denote by
\begin{align*} 
\phi_\ell(t)\doteq \frac{t^{\ell+1}}{\ell+1} 
\end{align*} the primitive of the speed of propagation $t^{\ell}$  vanishing at $t=0$. Hence, the function $$A_\ell(t)\doteq \phi_\ell(t)-\phi_\ell(1)$$ describes the amplitude of the light-cone for the Cauchy problem with data prescribed at the initial time $t_0=1$.
$B_R$ is the ball in $\mathbb{R}^n$ around the origin with radius $R$. By $f\lesssim g$ we mean that there exists a  constant $C>0$ such that $f\leqslant Cg$ and, similarly, for the notation $f\gtrsim g$.
As in the introduction, we denote by $p_{\mathrm{Fuj}}(n)$ the Fujita exponent and by $p_{\mathrm{Str}}(n,\ell)$ the Strauss-type exponent. 

\subsection{Main results}

Let us start by recalling the notion of weak solutions to  \eqref{semilinear EPDT} that has been employed in \cite{Pal21}. We point out that, even though we call them weak solutions, actually, for these solutions we ask more regularity than for traditional distributional solutions with respect to the time. This additional regularity is necessary to get a space average that is a $\mathcal{C}^2$ function with respect to t.

\begin{definition} \label{def sol} Let $u_0,u_1\in L^1_{\mathrm{loc}}(\mathbb{R}^n)$ such that $\mathrm{supp} \, u_0, \mathrm{supp} \, u_1 \subset B_R$ for some $R>0$. We say that $$u\in\mathcal{C}\big([1,T),W^{1,1}_{\mathrm{loc}}(\mathbb{R}^n)\big)\cap \mathcal{C}^1\big([1,T),L^{1}_{\mathrm{loc}}(\mathbb{R}^n)\big)\cap L^p_{\mathrm{loc}}\big((1,T)\times \mathbb{R}^n\big)$$ is a \emph{weak solution} to \eqref{semilinear EPDT} on $[1,T)$ if $u(1,\cdot)=\varepsilon u_0$ in $L^{1}_{\mathrm{loc}}(\mathbb{R}^n)$, $u$ fulfills the support condition 
\begin{align*}
\mathrm{supp} \, u(t,\cdot) \subset B_{R+\phi_\ell(t)-\phi_\ell(1)} \qquad \mbox{for any} \ t\in(1,T),
\end{align*} 
 and the integral identity
\begin{align}
& \int_{\mathbb{R}^n} \partial_t u(t,x)\phi(t,x) \, \mathrm{d}x+\int_1^t\int_{\mathbb{R}^n} \big(-\partial_t u(s,x)\phi_s(s,x)+s^{2\ell}\nabla u(s,x) \cdot \nabla \phi(s,x)\big)\mathrm{d}x\, \mathrm{d}s \notag  \\
 & \qquad + \int_1^t\int_{\mathbb{R}^n} \big(\mu s^{-1}\partial_t u(s,x)\phi(s,x)+\nu^2 s^{-2} u(s,x)  \phi(s,x)\big)\mathrm{d}x\, \mathrm{d}s \notag \\
& \quad = \varepsilon \int_{\mathbb{R}^n} u_1(x)\phi(1,x) \, \mathrm{d}x+ \int_1^t\int_{\mathbb{R}^n}|u(s,x)|^p\phi(s,x) \, \mathrm{d}x\, \mathrm{d}s \label{integral rel def sol}
\end{align} holds for any $t\in(1,T)$ and any test function $\phi\in\mathcal{C}^\infty_0\big([1,T)\times\mathbb{R}^n\big)$.
\end{definition}

We stress that in this paper, we will work with classical solutions to \eqref{semilinear EPDT}:  the reason is that we apply an integral representation formula in which we have to evaluate pointwisely the data and the nonlinearity. This being said, it is evident that classical solutions to \eqref{semilinear EPDT} are weak solutions (in the sense of  Definition \ref{def sol}) as well. Since we are going to use some inequalities from \cite{Pal21}, we had to recall the definition for weak solutions in the aforementioned sense, in spite of the fact that in our main result we work with classical solutions.

\begin{theorem}\label{Thm main} Let $\ell>-1$ and $\mu,\nu^2\geqslant 0$ such that $\delta\geqslant 0$. Let us assume that the exponent $p$ of the nonlinear term satisfies
\begin{align*}
p=p_{\mathrm{Str}}\!\left(n+\tfrac{\mu}{\ell+1},\ell\right).
\end{align*} 
Let $(u_0,u_1)\in \mathcal{C}^2_0(\mathbb{R}^n) \times \mathcal{C}^1_0(\mathbb{R}^n)$ be nonnegative,  nontrivial and compactly supported functions with supports contained in  $B_R$ for some $R>0$ such that $u_1+\frac{\mu-1-\sqrt{\delta}}{2}u_0\geqslant 0$ and
\begin{align}\label{integral assumption Cauchy data}
 \int_{\mathbb{R}^n} \left(u_1(x)+\tfrac{\mu-1+\sqrt{\delta}}{2}u_0(x) \right)\mathrm{d}x >0.
\end{align} 
Let $u\in \mathcal{C}^2\big([1,T) \times \mathbb{R}^n\big)$ be a classical solution to \eqref{semilinear EPDT} with lifespan $T=T(\varepsilon)$.

Then, there exists a positive constant $\varepsilon_0 = \varepsilon_0(u_0,u_1,n,p,\ell,\mu,\nu^2,R)$ such that for any $\varepsilon \in (0,\varepsilon_0]$ the weak solution $u$ blows up in finite time. Furthermore, the upper bound estimate for the lifespan
\begin{align}\label{lifespan est thm main}
T(\varepsilon)\leqslant \exp\left(E\varepsilon^{-p(p-1)}\right) 
\end{align}
holds, where the positive constant $E$ is independent of $\varepsilon$.
\end{theorem}

\begin{remark} The exponent $p_{\mathrm{Str}}\big(n+\tfrac{\mu}{\ell+1},\ell\big)$ is obtained from the Strauss-type exponent $p_{\mathrm{Str}}(n,\ell)$ defined through \eqref{crit exponent tricomi} by a shift of magnitude $\tfrac{\mu}{\ell+1}$ in the space dimension. In other words, $p_{\mathrm{Str}}\big(n+\tfrac{\mu}{\ell+1},\ell\big)$ is the larger root to the quadratic equation
\begin{align}\label{equation for p crit}
\left(\tfrac{n-1}{2}(\ell+1)+\tfrac{\ell+\mu}{2}\right) p^2 -\left(\tfrac{n+1}{2}(\ell+1)+\tfrac{\mu-3\ell}{2}\right) p -(\ell+1) =0.
\end{align}
\end{remark}

The proof of Theorem \ref{Thm main} follows the ideas in \cite{TakWak11,PalTak22} and is organized as follows: in Section \ref{Section ODI Kato} we prove a comparison argument for a second-order ODE, which generalizes Kato's lemma in the critical case; in Section \ref{Section Radon transform} we reduce the problem to the 1-dimensional case by means of Radon transform and we apply an explicit integral representation formula of Yagdjian-type recalled in Section \ref{Section Yagdjian integral representation}  to obtain an iteration frame for the $L^p(\mathbb{R}^n)$-norm of a local solution; finally, in Section \ref{Section iteration  argument} we derive a sequence of lower bound estimates for the above mentioned $L^p(\mathbb{R}^n)$-norm with additional logarithmic terms and then we use these estimates to prove the blow-up result by applying the Kato-type lemma in the critical case to the space average of a local solution.
\section{Preliminary results}

In this section, we recall some estimates that are proved in \cite[Section 2]{Pal21} and that we are going to use in the proof our main result.

\subsection{Iteration frame for the space average of a local solution}

Let $u\in\mathcal{C}^2\big([1,T)\times \mathbb{R}^n\big)$ a classical solution to \eqref{semilinear EPDT}. Let us denote by
\begin{align*}
U(t)\doteq \int_{\mathbb{R}^n} u(t,x)\, \mathrm{d}x \qquad \mbox{for} \ t\in [1,T) ,
\end{align*} the spatial average of $u$. In \cite[Section 2.1]{Pal21}, by the identity 
\begin{align}
U''(t) + \mu t^{-1} U'(t)+ \nu^2 t^{-2} U(t)= \int_{\mathbb{R}^n} |u(t,x)|^p \,  \mathrm{d}x \label{ODE for U}
\end{align}
it is easy to derive the following representation for $U$
\begin{align} \label{Representation for U}
U(t) &=U_{\mathrm{lin}}(t)  + t^{-r_1} \int_1^t s^{r_1-r_2-1} \int_1^s \tau^{r_2+1} \, \|u(\tau,\cdot)\|^p_{L^p(\mathbb{R}^n)} \, \mathrm{d}\tau \, \mathrm{d}s,
\end{align} 
where $r_1,r_2$ are the roots of the equation $r^2-(\mu-1)r+\nu^2=0$ and 
\begin{align*}
U_{\mathrm{lin}}(t) \doteq \begin{cases}
\displaystyle{\tfrac{r_1 t^{-r_2}-r_2 t^{-r_1}}{r_1-r_2}\, U(1)+\tfrac{t^{-r_2}-t^{-r_1}}{r_1-r_2} \, U'(1)} & \mbox{if} \ \ \delta>0, \\
t^{-r_1}(1+r_1\ln t)\, U(1)+t^{-r_1}\ln t \, U'(1) & \mbox{if} \ \ \delta=0.
\end{cases}
\end{align*}
Furthermore, by employing H\"older's inequality and the support condition for $u$, we have the following iteration frame for $U$
\begin{align*}
U(t) & \geqslant C t^{-r_1} \int_1^t s^{r_1-r_2-1} \int_1^s \tau^{r_2+1} \tau^{-n(\ell+1)(p-1)} (U(\tau))^p\, \mathrm{d}\tau \, \mathrm{d}s 
\end{align*} for a suitable positive constant $C$ that depends on $n,p,\ell$, see \cite[Eq. (21)]{Pal21}.

\subsection{First lower bound estimate for $\|u(t,\cdot)\|^p_{L^p(\mathbb{R}^n)}$}

In \cite[Subsection 2.2]{Pal21}, by studying the lower bound of  a weighted space average of $u$, we obtained the following lower bound estimate:
\begin{align}\label{1st low bound ||u||^p}
\|u(t,\cdot)\|^p_{L^p(\mathbb{R}^n)}\geqslant \widetilde{C} \varepsilon^p t^{-\left(\frac{n-1}{2}(\ell+1)+\frac{\ell+\mu}{2}\right)p+(n-1)(\ell+1)}
\end{align}
for $t\geqslant T_2$, where $\widetilde{C}$ is a positive constant depending on $n,\ell,\mu,\nu^2,R,u_0,u_1$ and $T_2=T_2(\mu,\nu^2,\ell)>2$.

\section{ODI comparison argument in the critical case}
\label{Section ODI Kato}

In this section we state and prove a Kato-type lemma for a second order ODI (\emph{ordinary differential inequality}) with special time-dependent coefficients. In particular, we consider a critical condition for the parameter appearing in the lower bound for the time-dependent function.

\begin{lemma}[Kato-type lemma in the critical scale-invariant case] \label{Lemma Kato-type}

Let $\mu,\nu^2$ be nonnegative real numbers such that $\delta=(\mu-1)^2-4\nu^2\geqslant 0$. We set 
\begin{align}\label{def r1,r2}
r_1\doteq \frac{1}{2}\big(\mu-1+\sqrt{\delta}\big) \quad \mbox{and} \quad r_2\doteq \frac{1}{2}\big(\mu-1-\sqrt{\delta}\big). 
\end{align}
Let $p>1$, $q\geqslant 0$ and $a\in\mathbb{R}$ such that
\begin{align}\label{critical condition a,p,q}
a(p-1)=q-2.
\end{align}
Assume that $G\in\mathcal{C}^2([1,T))$ fulfills
\begin{align}
& G''(t)+\mu t^{-1} G'(t)+\nu^2 t^{-2} G(t)\geqslant B t^{-q} |G(t)|^p && \mbox{for} \ t\in[1,T), \label{Kato lemma ODI G} \\
& G(t) \geqslant K t^a   && \mbox{for} \ t\in[T_0,T), \label{Kato lemma lower bound G} \\
& G(1), G'(1)\geqslant 0 \quad \mbox{and} \quad G'(1)+r_1G(1)>0, \label{Kato lemma initial condition G}
\end{align} for some $T_0\in [1,T)$ and $B,K>0$. Let us define $T_1\doteq \max\{T_0,\widetilde{T}_0\}$, where
\begin{align}\label{def T0 tilde}
\widetilde{T}_0=\widetilde{T}_0(\mu,\nu^2,G(1),G'(1))\doteq \begin{cases} \left(1+\frac{(r_1-r_2)G(1)}{G'(1)+r_1 G(1)}\right)^{\frac{1}{r_1-r_2}} & \mbox{for} \ r_1\neq r_2, \\
\exp\left(\frac{G(1)}{G'(1)+r_1 G(1)}\right) & \mbox{for} \ r_1=	 r_2,
 \end{cases}
\end{align}
 and
\begin{align}\label{def K0}
& K_0\doteq 
\begin{cases}
\left(\frac{p+1}{B}\right)^{\frac{1}{p-1}} \left(\frac{a+r_1}{1-2^{-\beta(a+r_1)}}\right)^{\frac{2}{p-1}}  & \mbox{if} \  a+r_1>0, \\
\left(\frac{p+1}{B}\right)^{\frac{1}{p-1}}  (\beta \ln 2)^{-\frac{2}{p-1}} & \mbox{if} \  a+r_1=0,
\end{cases}
\end{align} where $\beta\in \left(0,\frac{p-1}{2}\right)$ is an arbitrarily fixed parameter.

\noindent If the multiplicative constant in \eqref{Kato lemma lower bound G} satisfies $K\geqslant K_0$, than the lifespan of $G$ is finite and fulfills $T\leqslant 2 T_1$.
\end{lemma}

\begin{proof}
By contradiction we suppose that $G(t)$ is defined for any $t\in [1,2T_1]$, i.e. we are assuming that $T>2T_1$.
The first step of the proof consists in deriving a suitable factorization for the ordinary differential operator $\frac{\mathrm{d}^2}{\mathrm{d}^2 t}+\mu t^{-1}\frac{\mathrm{d}}{\mathrm{d} t}+\nu^2 t^{-2} $.
In particular, we look for two real constants $r_1,r_2$ such that
\begin{align*}
G''(t)+\mu t^{-1} G'(t)+\nu^2 t^{-2} G(t) 
&= t^{-(r_2+1)} \frac{\mathrm{d}}{\mathrm{d}t} \left(t^{r_2+1-r_1} \frac{\mathrm{d}}{\mathrm{d}t}\left(t^{r_1} G(t)\right)\right) \\
&= G''(t)+(r_1+r_2+1) t^{-1} G'(t)+r_1 r_2 t^{-2} G(t). 
\end{align*} Therefore, $r_1,r_2$ have to solve the system $$\begin{cases} r_1+r_2+1=\mu, \\ r_1 r_2=\nu^2,\end{cases}$$ that is, $r_1,r_2$ are the roots of the quadratic equation $r^2-(\mu-1)r+\nu^2=0$. Clearly, the role of $r_1,r_2$ is completely interchangeable, but for the sake of this proof it is convenient to fix $r_1,r_2$ as in \eqref{def r1,r2}. Let us introduce the auxiliary function $F(t)\doteq t^{r_1} G(t)$. From the previous considerations and \eqref{Kato lemma ODI G}, we see immediately that
\begin{align}\label{ODI F}
\frac{\mathrm{d}}{\mathrm{d}t} \left(t^{r_2+1-r_1} F'(t)\right)\geqslant B t^{r_2+1-q-r_1p} |F(t)|^p \quad \mbox{for} \ t\in[1,T).
\end{align} Integrating  \eqref{ODI F} over $[1,t]$ and using the nonnegativity of the nonlinear term, thanks to \eqref{Kato lemma initial condition G}, we obtain
\begin{align}\label{F' is positive}
F'(t) \geqslant t^{r_1-(r_2+1)} F'(1) = t^{r_1-(r_2+1)} (G'(1)+r_1G(1))>0
\end{align} for $t\in [1,T)$. 

Our next step in the proof will be to derive a lower bound estimate for $G$ that will provide in turns a lower bound estimate for $F$ as well. We underline that this lower bound estimate for $G$, differently from \eqref{Kato lemma lower bound G}, is valid for any $t\in[1,T)$. We begin by rewriting \eqref{Kato lemma ODI G} as 
\begin{align*}
t^{-(r_2+1)} \frac{\mathrm{d}}{\mathrm{d}t} \left(t^{r_2+1-r_1} \frac{\mathrm{d}}{\mathrm{d}t}\left(t^{r_1} G(t)\right)\right) \geqslant B t^{-q}|G(t)|^p
\end{align*} for $t\in[1,T)$. By straightforward intermediate computations we find
\begin{align*}
& t^{-(r_2+1)}  \frac{\mathrm{d}}{\mathrm{d}t} \left(t^{r_2+1-r_1} \frac{\mathrm{d}}{\mathrm{d}t}\left(t^{r_1} G(t)\right)\right) \geqslant 0 \qquad \mbox{for} \ \ t\in[1,T)\\
& \quad \Longrightarrow\qquad   t^{r_2+1-r_1} \frac{\mathrm{d}}{\mathrm{d}t}\left(t^{r_1} G(t)\right) -(G'(1)+r_1G(1))\geqslant 0 && \mbox{for} \ \ t\in[1,T)\\
& \quad \Longrightarrow\qquad    \frac{\mathrm{d}}{\mathrm{d}t}\left(t^{r_1} G(t)\right) \geqslant (G'(1)+r_1G(1)) \, t^{-(r_2+1)+r_1} && \mbox{for} \ \ t\in[1,T) \\
& \quad \Longrightarrow\qquad   t^{r_1} G(t) -G(1) \geqslant (G'(1)+r_1G(1)) \, \int_1^t s^{-(r_2+1)+r_1} \mathrm{d}s && \mbox{for} \ \ t\in[1,T).
\end{align*}
Let us continue the computations for $r_1\neq r_2 \Leftrightarrow \delta >0$:
\begin{align*}
G(t) &\geqslant t^{-r_1} G(1)+\frac{G'(1)+r_1G(1)}{r_1-r_2}\left(t^{-r_2}-t^{-r_1}\right)  \\
&= \frac{r_1 t^{-r_2}-r_2 t^{-r_1}}{r_1-r_2}G(1)+\frac{t^{-r_2}-t^{-r_1}}{r_1-r_2}G'(1)
\end{align*} for $t\in[1,T)$. Otherwise if $r_1= r_2 \Leftrightarrow \delta =0$, we have
\begin{align*}
G(t) &\geqslant t^{-r_1}(1+r_1\ln t)\, G(1)+t^{-r_1}\ln t \, G'(1)
\end{align*} for $t\in[1,T)$. Summarizing, we just showed that $G(t)\geqslant G_{\mathrm{lin}}(t)$ for any $t\in[1,T)$, where 
\begin{align*}
G_{\mathrm{lin}}(t)\doteq \begin{cases}
\displaystyle{\frac{r_1 t^{-r_2}-r_2 t^{-r_1}}{r_1-r_2}G(1)+\frac{t^{-r_2}-t^{-r_1}}{r_1-r_2}G'(1)} & \mbox{if} \ \ \delta>0, \\
t^{-r_1}(1+r_1\ln t)\, G(1)+t^{-r_1}\ln t \, G'(1) & \mbox{if} \ \ \delta=0.
\end{cases}
\end{align*}
Therefore, from this lower bound for $G$ and since $G(1),G'(1)\geqslant 0$ according to \eqref{Kato lemma initial condition G}, we get immediately that 
\begin{align*}
F(t)=t^{r_1}G(t)\geqslant t^{r_1} G_{\mathrm{lin}}(t) \geqslant 0
\end{align*}  for $t\in[1,T)$.
As a further consequence of \eqref{ODI F}, we have that
\begin{align*}
F''(t)+(r_2-r_1+1)\, t^{-1} F'(t)\geqslant B t^{-q-r_1(p-1)}(F(t))^p
\end{align*} for $t\in [1,T)$, where we used the fact that $F$ is nonnegative. In \eqref{F' is positive} we proved that $F'$ is always positive. Hence, by multiplying the previous inequality by $2F'(t)$, we have
\begin{align*}
\frac{\mathrm{d}}{\mathrm{d} t} (F'(t))^2+2(r_2-r_1+1)t^{-1} (F'(t))^2\geqslant \frac{2B}{p+1} t^{-q-r_1(p-1)}\frac{\mathrm{d}}{\mathrm{d}t}(F(t))^{p+1}
\end{align*} for $t\in [1,T)$. Next, we multiply the previous estimate by $t^{2(r_2-r_1+1)}$, arriving at
\begin{align}\label{estimate for F' ^2 before integration}
\frac{\mathrm{d}}{\mathrm{d} t}  \left(t^{2(r_2-r_1+1)} (F'(t))^2\right)\geqslant \frac{2B}{p+1} t^{-q-r_1(p-1)+2(r_2-r_1+1)}\frac{\mathrm{d}}{\mathrm{d}t}(F(t))^{p+1}
\end{align} for $t\in [1,T)$. We remark that the critical condition \eqref{critical condition a,p,q} implies 
\begin{align*}
-q-r_1(p-1)+2(r_2-r_1+1) 
= -(a+r_1)(p-1)+2(r_2-r_1)   =
 -(a+r_1)(p-1) -\sqrt{\delta}\leqslant 0.
\end{align*} In the previous inequality we used implicitly the fact that from the estimate $G(t)\geqslant G_{\mathrm{lin}}(t)$ we may assume without loss of generality that $a\geqslant \max\{-r_1,-r_2\}$. Consequently, integrating \eqref{estimate for F' ^2 before integration} over $[1,t]$ we get
\begin{align}
t^{2(r_2-r_1+1)} (F'(t))^2 -(F'(1))^2 &\geqslant  \frac{2B}{p+1} \int_1^t s^{-q-r_1(p-1)+2(r_2-r_1+1)}\frac{\mathrm{d}}{\mathrm{d}s}(F(s))^{p+1} \mathrm{d}s \notag\\
&\geqslant  \frac{2B}{p+1} t^{-q-r_1(p-1)+2(r_2-r_1+1)} \int_1^t \frac{\mathrm{d}}{\mathrm{d}s}(F(s))^{p+1} \mathrm{d}s \notag \\ 
&=\frac{2B}{p+1} t^{-q-r_1(p-1)+2(r_2-r_1+1)} \left((F(t))^{p+1} -F(1))^{p+1} \right) \label{estimate for F' ^2 after integration}
\end{align} for $t\in [1,T)$.
By elementary computations we have
\begin{align*}
F(t)-F(1)& =t^{r_1}G(t)-G(1)\\ 
&\geqslant t^{r_1}G_{\mathrm{lin}}(t)-G(1)  =\begin{cases}   
(G'(1)+r_1G(1)) \frac{t^{r_1-r_2}-1}{r_1-r_2} & \mbox{for} \ r_1\neq r_2, \\
(G'(1)+r_1G(1)) \ln t & \mbox{for} \ r_1 =r_2.
\end{cases}
\end{align*} Thus, $F(t)\geqslant F(1)$ for $t\in [1,T)$, which implies 
\begin{align}\label{estimate F^p+1(t)- F^p+1(0)}
(F(t))^{p+1} -(F(1))^{p+1} \geqslant (F(t))^{p}\left(F(t)-F(1)\right)
\end{align}  for $t\in [1,T)$. Our next goal is to determine for which $t$ the inequality $F(t)\geqslant 2 F(1)$ holds. Since 
\begin{align*}
F(t)-2 F(1) \geqslant t^{r_1}G_{\mathrm{lin}}(t) -2G(1)  =\begin{cases}   
(G'(1)+r_1G(1)) \frac{t^{r_1-r_2}-1}{r_1-r_2} -G(1) & \mbox{for} \ r_1\neq r_2, \\
(G'(1)+r_1G(1)) \ln t  -G(1) & \mbox{for} \ r_1 =r_2,
\end{cases}
\end{align*} the right-hand side of the previous inequality is nonnegative if and only if
\begin{align*}
  \frac{t^{r_1-r_2}-1}{r_1-r_2}\geqslant \frac{G(1)}{G'(1)+r_1G(1)} & \quad \Longleftrightarrow \quad t\geqslant \left(1+\frac{(r_1-r_2)G(1)}{G'(1)+r_1 G(1)}\right)^{\frac{1}{r_1-r_2}} && \mbox{for} \ \  r_1\neq r_2, \\
  \ln t\geqslant \frac{G(1)}{G'(1)+r_1G(1)} &\quad \Longleftrightarrow \quad t\geqslant \exp\left(\frac{(r_1-r_2)G(1)}{G'(1)+r_1 G(1)}\right) && \mbox{for} \ \  r_1= r_2. 
\end{align*} Summarizing, for $t\in  [\widetilde{T}_0,T)$, where $\widetilde{T}_0$ is defined in \eqref{def T0 tilde}, it results in $F(t)-F(1)\geqslant \frac{1}{2} F(t)$ and then, by \eqref{estimate F^p+1(t)- F^p+1(0)}, $(F(t))^{p+1} -(F(1))^{p+1} \geqslant \frac{1}{2} (F(t))^{p+1}$. Combining this last inequality with \eqref{estimate for F' ^2 after integration}, we get
\begin{align}\label{estimate intermediate Kato lemma}
(F'(t))^2 & \geqslant t^{2(r_1-(r_2+1))}(F'(1))^2+\frac{B}{p+1}t^{-q-r_1(p-1)}(F(t))^{p+1} \\
& \geqslant \frac{B}{p+1}t^{-q-r_1(p-1)}(F(t))^{p+1} \notag
\end{align} for $t\in [\widetilde{T}_0,T)$. Hence,
\begin{align}\label{estimate lb F'}
F'(t) &  \geqslant \sqrt{\frac{B}{p+1}} \, t^{-\frac{q}{2}-\frac{r_1}{2}(p-1)}(F(t))^{\frac{p+1}{2}},
\end{align} for $t\in [\widetilde{T}_0,T)$. Finally, we fix $\beta\in(0,\frac{p-1}{2})$. Multiplying \eqref{estimate lb F'} by $(F(t))^{-(2+\beta)}$, we find
\begin{align*}
\frac{F'(t)}{(F(t))^{1+\beta}}=-\frac{1}{\beta}\frac{\mathrm{d}}{\mathrm{d}t}\frac{1}{(F(t))^{\beta}}\geqslant \sqrt{\frac{B}{p+1}} \, t^{-\frac{q}{2}-\frac{r_1}{2}(p-1)}(F(t))^{\frac{p-1}{2}-\beta},
\end{align*} $t\in [\widetilde{T}_0,T)$. As in the statement, we set $T_1\doteq \max\{T_0, \widetilde{T}_0\}$. Integrating the last inequality over $[T_1,t]$, it follows that
\begin{align}\label{estimate final lemma}
\frac{1}{\beta}\left(\frac{1}{(F(T_1))^{\beta}}-\frac{1}{(F(t))^{\beta}}\right) \geqslant \sqrt{\frac{B}{p+1}} \int_{T_1}^t s^{-\frac{q}{2}-\frac{r_1}{2}(p-1)}(F(s))^{\frac{p-1}{2}-\beta} \mathrm{d}s
\end{align} for $t\in [T_1,T)$. Next, we are going to use the lower bound estimate \eqref{Kato lemma lower bound G}. Indeed, \eqref{Kato lemma lower bound G} implies 
\begin{align}\label{estimate lb F}
F(t)\geqslant K t^{a+r_1}
\end{align} for any $t\in [T_0,T)$. As we previously pointed out, we can assume that $a+r_1\geqslant 0$, due to $G(t)\geqslant G_{\mathrm{lin}}(t)$. In order to complete the proof, we plug \eqref{estimate lb F} into \eqref{estimate final lemma}. We consider separately the case $a+r_1>0$ from the case $a+r_1=0$. For $a+r_1>0$, using \eqref{critical condition a,p,q} and the lower bound estimate \eqref{estimate lb F} twice in \eqref{estimate final lemma}, we get
\begin{align*}
\frac{K^{-\beta}}{\beta}T_1^{-\beta(a+r_1)} & \geqslant \frac{1}{\beta (F(T_1))^\beta} > \frac{1}{\beta}\left(\frac{1}{(F(T_1))^{\beta}}-\frac{1}{(F(t))^{\beta}}\right) \\ & \geqslant \sqrt{\frac{B}{p+1}} \int_{T_1}^t s^{-\frac{q}{2}-\frac{r_1}{2}(p-1)}(F(s))^{\frac{p-1}{2}-\beta} \mathrm{d}s \\
 & \geqslant \sqrt{\frac{B}{p+1}} K^{\frac{p-1}{2}-\beta} \int_{T_1}^t s^{-\frac{q}{2}+\frac{a}{2}(p-1)-\beta(a+r_1)} \mathrm{d}s \\ & = \sqrt{\frac{B}{p+1}} K^{\frac{p-1}{2}-\beta} \int_{T_1}^t s^{-1-\beta(a+r_1)} \mathrm{d}s \\ & = \sqrt{\frac{B}{p+1}} \frac{K^{\frac{p-1}{2}-\beta}}{\beta(a+r_1)}\left(T_1^{-\beta(a+r_1)}-t^{-\beta(a+r_1)}\right)
\end{align*} for $t\in [T_1,T)$. In particular, picking $t=2T_1$ in the previous estimate, it results in
\begin{align*}
\frac{a+r_1}{1-2^{-\beta(a+r_1)}}\sqrt{\frac{p+1}{B}}> K^{\frac{p-1}{2}}
\end{align*} which contradicts the assumption $K\geqslant K_0$, where the definition of $K_0$ is given in \eqref{def K0}. Similarly, for $a+r_1=0$ we have
\begin{align*}
\frac{K^{-\beta}}{\beta} 
 & > \sqrt{\frac{B}{p+1}} K^{\frac{p-1}{2}-\beta} \int_{T_1}^t s^{-\frac{q}{2}+\frac{a}{2}(p-1)} \mathrm{d}s \\ & = \sqrt{\frac{B}{p+1}} K^{\frac{p-1}{2}-\beta} \int_{T_1}^t s^{-1} \mathrm{d}s \\ & = \sqrt{\frac{B}{p+1}} K^{\frac{p-1}{2}-\beta}\ln \frac{t}{T_1}
\end{align*} for $t\in [T_1,T)$. Choosing $t=2T_1$ also in this case, we derive 
\begin{align*}
\frac{1}{\beta \ln 2}\sqrt{\frac{p+1}{B}}>K^{\frac{p-1}{2}},
\end{align*} which contradicts the assumption $K\geqslant K_0$ even in this case. The proof is completed.
\end{proof}

\begin{remark} The comparison result in Lemma \ref{Lemma Kato-type} generalizes the classical Kato's lemma (see \cite[Lemma 2.1]{Tak15}) in the critical case (cf. \cite[Lemma 2.1]{YZ06} and \cite[Lemma 2.1]{TakWak11}) which is recovered for $\mu=\nu^2=0$ (up to a translation in the initial time). We emphasize that the critical condition for $a,p,q$ in \eqref{critical condition a,p,q} is exactly the same one as for classical Kato's lemma. 
\end{remark}

\begin{remark}
In the intermediate estimate \eqref{estimate intermediate Kato lemma} above we neglected the term $t^{2(r_1-(r_2+1))}(F'(1))^2$ in the lower bound for $(F'(t))^2$. We underline explicitly that this term would allow to obtain the same kind of lower bound estimate for $F(t)$ as in the estimate $F(t)\geqslant t^{r_1}G_{\mathrm{lin}}(t)$ that we actually used in the proof.
\end{remark}

\section{Integral representation formula for the 1D linear Cauchy problem} \label{Section Yagdjian integral representation}

We recall now an integral representation formula for the solution of the following linear Cauchy problem in one space dimension
\begin{align}\label{linear EPDT 1d}
\begin{cases} \partial_t^2u-t^{2\ell}\partial_x^2 u + \mu t^{-1} \partial_t u+\nu^2 t^{-2} u=g(t,x), &  x\in \mathbb{R}, \ t>1,\\
u(1,x)=u_0(x), & x\in \mathbb{R}, \\  u_t(1,x)= u_1(x), & x\in \mathbb{R}.
\end{cases}
\end{align} 
This representation has been derived by using \emph{Yagdjian's integral transform approach} (cf.  \cite{Yag04,Yag06,Yag07,Yag10,Yag15}) in \cite{Pal19RF} for the scale invariant wave model (that is, for $\ell=0$) and in \cite[Theorem 2.1]{HHP20} by means of a suitable change of variables for $\ell\in (-1,0)$. The same proof used in \cite{HHP20} can be actually repeated verbatim to derive the representation in the more general case $\ell>-1$, obtaining the result in the next lemma.

\begin{lemma}\label{Lemma RF} Let $n=1$, $\ell>-1$ and $\mu,\nu^2\geqslant 0$ such that $\delta\geqslant 0$. Let us assume that $u_0\in \mathcal{C}^2(\mathbb{R})$, $u_1\in\mathcal{C}^1(\mathbb{R})$ and $g\in\mathcal{C}\big([0,\infty),\mathcal{C}^1(\mathbb{R})\big)$. Then, the classical solution to \eqref{linear EPDT 1d} admits the following representation
\begin{align}
u(t,x) & =\frac{1}{2}\,  t^{-\frac{\mu+\ell}{2}}\left[u_0(x+A_\ell(t))+u_0(x-A_\ell(t))\right] 
+  \int_{x-A_\ell(t)}^{x+A_\ell(t)} u_0(y) K_0(t,x;y;\mu,\nu^2,\ell) \mathrm{d}y \notag \\
& \quad +  \int_{x-A_\ell(t)}^{x+A_\ell(t)} u_1(y) K_1(t,x;y;\mu,\nu^2,\ell) \mathrm{d}y  + \int_{1}^t\int_{x-A_\ell(t)+A_\ell(b)}^{x+A_\ell(t)-A_\ell(b)} g(b,y) E(t,x;b,y;\mu,\nu^2,\ell) \mathrm{d}y \, \mathrm{d}b .
\label{integral repr lin 1D}
\end{align} Here the kernel function $E$ is defined by
\begin{align}
E=E(t,x;b,y;\mu,\nu^2,\ell) & \doteq c \, t^{-\frac{\mu}{2}+\frac{1-\sqrt{\delta}}{2}} b^{\frac{\mu}{2}+\frac{1-\sqrt{\delta}}{2}} \left((\phi_\ell(t)+\phi_\ell(b))^2-(y-x)^2\right)^{-\gamma} \notag  \\ & \qquad \times F\left(\gamma,\gamma;1;\frac{(\phi_\ell(t)-\phi_\ell(b))^2-(y-x)^2}{(\phi_\ell(t)+\phi_\ell(b))^2-(y-x)^2}\right), \label{def kernel E}
\end{align} where 
\begin{align*}
c=c(\mu,\nu^2,\ell) &\doteq 2^{-\frac{\sqrt{\delta}}{1+\ell}} (1+\ell)^{-1+\frac{\sqrt{\delta}}{1+\ell}}, \\
\gamma=\gamma(\mu,\nu^2,\ell)  &\doteq \frac{1}{2}-\frac{\sqrt{\delta}}{2(1+\ell)},
\end{align*} and $F(\alpha_1,\alpha_2,\beta_1,z)$ denotes the Gauss hypergeometric function, while the kernel functions $K_0,K_1$ appearing in the integral terms involving the Cauchy data are given by
\begin{align}
K_0=K_0(t,x;y;\mu,\nu^2,\ell) & \doteq \mu E(t,x;1,y;\mu,\nu^2,\ell)-\frac{\partial E}{\partial b} (t,x;b,y;\mu,\nu^2,\ell) \Big|_{b=1}, \label{def kernel K0} \\
K_1=K_1(t,x;y;\mu,\nu^2,\ell) & \doteq E(t,x;1,y;\mu,\nu^2,\ell).\label{def kernel K1}
\end{align}
\end{lemma}

\begin{remark}\label{remark K0} Clearly, the kernels $E$ and $K_1$ are nonnegative on the corresponding domains of integration in \eqref{integral repr lin 1D}. Let us determine a more explicit representation for the kernel $K_0$.
Thanks to the property of the hypergeometric function
\begin{align*}
F(\alpha_1,\alpha_2;\beta;\zeta)=(1-\zeta)^{\beta-(\alpha_1+\alpha_2)} F(\beta-\alpha_1,\beta-\alpha_2;\beta;\zeta), 
\end{align*} see cf. \cite[Eq.(15.8.1)]{OLBC10}), and the relation $1-2\gamma=\frac{\sqrt{\delta}}{\ell+1}$ we can rewrite the kernel $E$ as follows
\begin{align*}
& \hphantom{=} E (t,x;b,y;\mu,\nu^2,\ell) \\
 &= c \, t^{-\frac{\mu}{2}+\frac{1-\sqrt{\delta}}{2}} b^{\frac{\mu}{2}+\frac{1-\sqrt{\delta}}{2}} \left((\phi_\ell(t)+\phi_\ell(b))^2-(y-x)^2\right)^{-1+\gamma}  \left(4 \phi_\ell(b)\phi_\ell(b)\right)^{1-2\gamma}  F\big(1-\gamma,1-\gamma;1;z(t,x;b,y;\ell)\big) \\
 &= \left(\tfrac{2}{\ell+1}\right)^{2(1-2\gamma)} c \, t^{-\frac{\mu}{2}+\frac{1+\sqrt{\delta}}{2}} b^{\frac{\mu}{2}+\frac{1+\sqrt{\delta}}{2}} \left((\phi_\ell(t)+\phi_\ell(b))^2-(y-x)^2\right)^{\gamma-1}  F\big(1-\gamma,1-\gamma;1;z(t,x;b,y;\ell)\big), 
\end{align*} where $$z(t,x;b,y;\ell)\doteq \frac{(\phi_\ell(t)-\phi_\ell(b))^2-(y-x)^2}{(\phi_\ell(t)+\phi_\ell(b))^2-(y-x)^2}.$$
For the ease of notation, we introduce the function 
\begin{align*}
\mathcal{E}(t,x;b,y;\mu,\nu^2,\ell)\doteq b^{\frac{\mu}{2}+\frac{1+\sqrt{\delta}}{2}} \left((\phi_\ell(t)+\phi_\ell(b))^2-(y-x)^2\right)^{-1+\gamma}  F\big(1-\gamma,1-\gamma;1;z(t,x;b,y;\ell)\big). 
\end{align*} By using the recursive identity $\frac{\mathcal{d}F}{\mathcal{d}\zeta}(\alpha_1,\alpha_2;\beta;\zeta)=\frac{\alpha_1 \alpha_2}{\beta}F(\alpha_1+1,\alpha_2+1;\beta+1;\zeta)$, cf. \cite[Eq.(15.5.1)]{OLBC10}, we obtain
\begin{align*}
& \frac{\partial \mathcal{E}}{\partial b}(t,x;b,y;\mu,\nu^2,\ell) \\ & = b^{\frac{\mu}{2}+\frac{1+\sqrt{\delta}}{2}} \left((\phi_\ell(t)+\phi_\ell(b))^2-(y-x)^2\right)^{\gamma-1} \\ 
&  \quad\times \bigg\{  \left[\left(\tfrac{\mu+1+\sqrt{\delta}}{2}\right)\, b^{-1}+\frac{2(\gamma-1) b^{\ell}\, (\phi_\ell(t)+\phi_\ell(b)) }{(\phi_\ell(t)+\phi_\ell(b)^2-(y-x)^2}\right] F\big(1-\gamma,1-\gamma;1; z(t,x;b,y;\ell)\big) \\ & \qquad \quad +(1-\gamma)^2 \partial_b z (t,x;b,y;\ell) \, F\big(2-\gamma,2-\gamma;2; z(t,x;b,y;\ell)\big)\bigg\}.
\end{align*} Therefore,
\begin{align}
K_0(t,x;y;\mu,\nu^2,\ell) & =  \left(\tfrac{2}{\ell+1}\right)^{2(1-2\gamma)} c\, t^{-\frac{\mu}{2}+\frac{1+\sqrt{\delta}}{2}} \left(\mu-\frac{\partial}{\partial b}\Big|_{b=1}\right) \mathcal{E}(t,x;b,y;\mu,\nu^2,\ell)  \notag \\ &= \left(\tfrac{2}{\ell+1}\right)^{2(1-2\gamma)} c\, t^{-\frac{\mu}{2}+\frac{1+\sqrt{\delta}}{2}} \left((\phi_\ell(t)+\phi_\ell(1))^2-(y-x)^2\right)^{\gamma-1} \notag \\ & \quad \times \Big\{  \Big[\tfrac{\mu-1-\sqrt{\delta}}{2} +\frac{2(1-\gamma) (\phi_\ell(t)+\phi_\ell(1))}{(\phi_\ell(t)+\phi_\ell(1)^2-(y-x)^2}\Big] F\big(1-\gamma,1-\gamma;1; z(t,x;1,y;\ell)\big) \notag \\ & \qquad \quad -(1-\gamma)^2 \partial_b z(t,x;1,y;\ell) \, F\big(2-\gamma,2-\gamma;2; z(t,x;1,y;\ell)\big)\Big\}.\label{explicit K0}
\end{align} By elementary computations, for $b\in [1,t]$ and $y\in[x-\phi_\ell(t)+\phi_\ell(b),x+\phi_\ell(t)-\phi_\ell(b)]$ we find that 
\begin{align*}
\frac{\partial z}{\partial b}(t,x;1,y;\ell)
&= -\frac{2b^{\ell} \phi_\ell(t) [\phi^2_\ell(t)-\phi^2_\ell(b)-(y-x)^2]}{[(\phi_\ell(t)+\phi_\ell(b))^2-(y-x)^2]^2}\leqslant 0,
\end{align*} since  $\phi^2_\ell(t)-\phi^2_\ell(b)-(y-x)^2\geqslant 2\phi_\ell(b)(\phi_\ell(t)-\phi_\ell(b))\geqslant 0$. Furthermore, the constant $1-\gamma =\frac{1}{2}(1+\frac{\sqrt{\delta}}{\ell+1})$ is always positive. Recalling \eqref{def delta}, we see that for $\mu\geqslant 1$ the constant $\frac{\mu-1-\sqrt{\delta}}{2}$ is nonnegative, while for $\mu\in[0,1)$ is strictly negative. Summarizing, for $\mu\geqslant 1$ the kernel function $K_0(t,x;y;\mu,\nu^2,\ell)$ is nonnegative for $y\in[x-A_\ell(t),x+A_\ell(t)]$, while for $\mu\in [0,1)$ the first of the three terms in the curl brackets is negative and the other two are nonnegative.

In the next section, in order to have a nonnegative solution to the 1-dimensional homogeneous problem when $\mu\in[0,1)$, we shall require that the Cauchy data are not only nonnegative, but they also satisfy the condition $u_1+\frac{\mu-1-\sqrt{\delta}}{2}u_0\geqslant 0$ (of course, when $\mu\geqslant 1$ the nonnegativity of the data is sufficient for this purpose, being the kernel functions $K_0,K_1$ nonnegative).
\end{remark}

\section{Iteration frame for $\|u(t,\cdot)\|^p_{L^p(\mathbb{R}^n)}$ via the Radon transform}
\label{Section Radon transform}

In this section we are going to derive an iteration frame for $\|u(t,\cdot)\|^p_{L^p(\mathbb{R}^n)}$. Following the idea from \cite{YZ06} we first reduce the problem to one space dimension by means of the Radon transform. 

We begin by remarking that we may assume without loss of generality that our solution $u$ is radially symmetric (otherwise, since the kernel $E$ is nonnegative, one can work with its spherical means instead, cf. \cite[Step 3 of the proof of Lemma 2.2]{YZ06}).

When $n\geqslant 2$ we consider the Radon transform $\mathcal{R}[u]$, with respect to the space variable, of the local classical solution $u=u(t,x)$ to \eqref{semilinear EPDT}. For $n=1$ everything that we are going to prove for $\mathcal{R}[u]$ can be simply proved by working directly with $u$.

Given $\varrho\in\mathbb{R}$ and $\xi\in\mathbb{R}^n$ such that $|\xi|=1$, we recall that the \emph{Radon transform} of $u(t,\cdot)$ is defined as
\begin{align}
\mathcal{R}[u](t,\varrho,\xi) &\doteq \int_{\{x\in\mathbb{R}^n| x\cdot \xi=\varrho\}} u(t,x) \mathrm{d}\sigma_x \notag \\
&= \int_{\{x\in\mathbb{R}^n| x\cdot \xi=0\}} u(t,x+\varrho \xi) \mathrm{d}\sigma_x, \label{def Radon transf}
\end{align} where $\mathrm{d}\sigma_x$ denote the Lebesgue measure on the hyperplanes. Since $u(t,\cdot)$ is radially symmetric, it turns out that $\mathcal{R}[u]$ is independent of $\xi$ and admits the representation
\begin{align*}
\mathcal{R}[u](t,\varrho)=\omega_{n-1}\int_{|\varrho|}^{\infty}u(t,r)(r^2-\varrho^2)^\frac{n-3}{2} r\mathrm{d}r,
\end{align*} where $\omega_{n-1}$ is the measure of the unit ball in $\mathbb{R}^{n-1}$. 

Thanks to the identity $\mathcal{R}[\Delta u](t,\varrho)=\partial_\varrho^2 \mathcal{R}[u](t,\varrho)$ (cf. \cite[Lemma 2.1, Chap. 1]{Hel11}), we see that $\mathcal{R}[u](t,\varrho)$ is a solution to the following Cauchy problem
\begin{align*} 
\begin{cases} \partial_t^2\mathcal{R}[u]-t^{2\ell}\partial_\varrho^2 \mathcal{R}[u] + \mu t^{-1} \partial_t \mathcal{R}[u]+\nu^2 t^{-2} \mathcal{R}[u]=\mathcal{R}[|u|^p], &  \varrho\in \mathbb{R}, \ t>1,\\
\mathcal{R}[u](1,\varrho)=\varepsilon\mathcal{R}[u_0](\varrho), & \varrho\in \mathbb{R}, \\  \partial_t \mathcal{R}[u](1,\varrho)=\varepsilon \mathcal{R}[u_1](\varrho), & \varrho\in \mathbb{R}.
\end{cases}
\end{align*} 
From Lemma \ref{Lemma RF}, we have the
representation
\begin{align*}
\mathcal{R}[u](t,\varrho) = \varepsilon (\mathcal{R}[u])_{\mathrm{lin}}(t,\varrho)+ (\mathcal{R}[u])_{\mathrm{nlin}}(t,\varrho),
\end{align*} where
\begin{align*}
(\mathcal{R}[u])_{\mathrm{lin}}(t,\varrho)&\doteq  \tfrac{1}{2}\,  t^{-\frac{\mu+\ell}{2}}\left(\mathcal{R}[u_0](\varrho+A_\ell(t))+\mathcal{R}[u_0](\varrho-A_\ell(t))\right) \\ 
& \quad +\int_{\varrho-A_\ell(t)}^{\varrho+A_\ell(t)} \mathcal{R}[u_0](\eta) K_0(t,\varrho;\eta;\mu,\nu^2,\ell) \mathrm{d}\eta \notag \\
& \quad +  \int_{\varrho-A_\ell(t)}^{\varrho+A_\ell(t)} \mathcal{R}[u_1](\eta) K_1(t,\varrho;\eta;\mu,\nu^2,\ell) \mathrm{d}\eta , \\ 
(\mathcal{R}[u])_{\mathrm{nlin}}(t,\varrho)&\doteq  \int_{1}^t\int_{\varrho-A_\ell(t)+A_\ell(b)}^{\varrho+A_\ell(t)-A_\ell(b)} \mathcal{R}[|u|^p](b,\eta) E(t,\varrho;b,\eta;\mu,\nu^2,\ell) \mathrm{d}\eta \, \mathrm{d}b.
\end{align*}

Thanks to the assumptions on the Cauchy data in Theorem \ref{Thm main}, we have that $(\mathcal{R}[u])_{\mathrm{lin}}$ is a nonnegative function. In particular, according to Remark \ref{remark K0}, when $\mu\geqslant 1$ this follows from the fact that $u_0,u_1$ are nonnegative, while for $\mu\in[0,1)$ we employ the condition $u_1+\frac{\mu-1-\sqrt{\delta}}{2}u_0\geqslant 0$ as well. Indeed, from \eqref{explicit K0} we find that in this latter case the function $\mathcal{R}[u_0](\eta) K_0(t,\varrho;\eta;\mu,\nu^2,\ell)+\mathcal{R}[u_1](\eta) K_1(t,\varrho;\eta;\mu,\nu^2,\ell)$ is nonnegative over the interval $[\varrho-A_\ell(t),\varrho+A_\ell(t)]$. 

 Consequently,
\begin{align}\label{Radon transf ineq 1}
\mathcal{R}[u](t,\varrho) \geqslant  (\mathcal{R}[u])_{\mathrm{nlin}}(t,\varrho).
\end{align}
From the support condition $\mathrm{supp}\, u_0,\mathrm{supp}\, u_1\subset B_R$ it follows that $\mathrm{supp} \, u(t,\cdot)\subset B_{R+A_\ell(t)}$ for any $t\in[1,T)$. Therefore, $$\mathcal{R}[u](t,\cdot)\subset [-(R+A_\ell(t)),R+A_\ell(t)]$$ for any $t\in[1,T)$. Indeed, for $|\varrho|>R+A_\ell(t)$ from \eqref{def Radon transf} on the hyperplane $\{x\in\mathbb{R}^n| x\cdot \omega=0\}$ we have
\begin{align*}
|\varrho \xi+x|^2=\varrho^2+|x|^2\geqslant \varrho^2 \quad \Longrightarrow \quad |\varrho \xi+x| > R+A_\ell(t) \quad \Longrightarrow \quad u(t,\varrho \xi+x)=0
\end{align*} and then $\mathcal{R}[u](t,\varrho)=0$ (the integrand function on the above hyperplane is identically 0). For the same reasons, it holds that $\mathrm{supp }\, \mathcal{R}[|u|^p](t,\cdot)\subset [-(R+A_\ell(t)),R+A_\ell(t)]$ for any $t\in [1,T)$. We point out that 
\begin{align*}
[-(R+A_\ell(b)),R+A_\ell(b)] \subset [\varrho-A_\ell(t)+A_\ell(b), \varrho+A_\ell(t)-A_\ell(b)] \ \ \Longleftrightarrow \ \ b\leqslant A_\ell^{-1}\left(\tfrac{1}{2}(A_\ell(t)-|\varrho|-R)\right).
\end{align*} Therefore, we introduce the quantity $b_0\doteq A_\ell^{-1}\left(\tfrac{1}{2}(A_\ell(t)-|\varrho|-R)\right)$. We stress that $b_0\geqslant 1$ if and only if $|\varrho|\leqslant A_\ell(t)-R$. From \eqref{Radon transf ineq 1} we have 
\begin{align}
\mathcal{R}[u](t,\varrho) & \geqslant  \int_{1}^{b_0}\int_{\varrho-A_\ell(t)+A_\ell(b)}^{\varrho+A_\ell(t)-A_\ell(b)} \mathcal{R}[|u|^p](b,\eta) E(t,\varrho;b,\eta;\mu,\nu^2,\ell) \mathrm{d}\eta \, \mathrm{d}b \notag \\
& \geqslant  \int_{1}^{b_0}\int_{-(R+A_\ell(b))}^{R+A_\ell(b)} \mathcal{R}[|u|^p](b,\eta) E(t,\varrho;b,\eta;\mu,\nu^2,\ell) \mathrm{d}\eta \, \mathrm{d}b \label{Radon transf ineq 2}
\end{align}
for $|\varrho|\leqslant A_\ell(t)-R$. Now we want to derive a lower bound estimate for the kernel function $E$ when $b\in [1,b_0]$ and $\eta\in [-(R+A_\ell(b)),R+A_\ell(b)].$ Let us recall that
\begin{align}\label{recall def E}
E(t,\varrho;b,\eta;\mu,\nu^2,\ell)= c t^{-\frac{\mu}{2}+\frac{1-\sqrt{\delta}}{2}}b^{\frac{\mu}{2}+\frac{1-\sqrt{\delta}}{2}}\left((\phi_\ell(t)+\phi_\ell(b))^2-(\varrho-\eta)^2\right)^{-\gamma} F(\gamma,\gamma;1;\zeta),
\end{align} where 
\begin{align*}
\gamma=\tfrac{1}{2}\left(1-\tfrac{\sqrt{\delta}}{\ell+1}\right) \quad \mbox{and} \quad \zeta=\zeta(t,\varrho;b,\eta;\ell)\doteq \frac{(\phi_\ell(t)-\phi_\ell(b))^2-(\varrho-\eta)^2}{(\phi_\ell(t)+\phi_\ell(b))^2-(\varrho-\eta)^2}.
\end{align*} Thanks to the Taylor expansion for the hypergeometric function
\begin{align*}
F(\gamma,\gamma;1;z) =\sum_{k=0}^{\infty} \frac{(\gamma)^2_k}{(1)_k \, k!} z^k,
\end{align*} where $(\gamma)_k$ is the so-called Pochhammer symbol for the rising factorial, we immediately see that $F(\gamma,\gamma;1;\zeta)\geqslant 1$ for any $\zeta\in [0,1)$.

Hence, it remains to estimate from below the factor $\left((\phi_\ell(t)+\phi_\ell(b))^2-(\varrho-\eta)^2\right)^{-\gamma}$ in \eqref{recall def E}. Since $\gamma$ changes sign depending on the rage for the parameters $\mu,\nu^2,\ell$, we need to consider two subcases (that is, the cases $\gamma\geqslant 0$ and  $\gamma\leqslant 0$).

Let us begin by showing that $b\in [1,b_0]$ and $\eta\in [-(R+A_\ell(b)),R+A_\ell(b)]$ the following inequalities hold
\begin{align}
\phi_\ell(t)+\phi_\ell(b)-\varrho +\eta\leqslant 2(\phi_\ell(t)-\varrho),\label{kernel estimate 1} \\
\phi_\ell(t)+\phi_\ell(b)+\varrho -\eta\leqslant 2(\phi_\ell(t)+\varrho)\label{kernel estimate 2}.
\end{align} We remark that \eqref{kernel estimate 1} is equivalent to require that $\phi_\ell(b)\leqslant \phi_\ell(t)-\varrho-\eta$. In order to check \eqref{kernel estimate 1}, we prove  the previous inequality is satisfied. Therefore, for $b\in [1,b_0]$ and $\eta\in [-(R+A_\ell(b)),R+A_\ell(b)]$ 
\begin{align*}
\phi_\ell(t)-\varrho-\eta & \geqslant \phi_\ell(t)-\varrho-A_\ell(b) -R \geqslant \phi_\ell(t)-\varrho-A_\ell(b_0) -R \\ &= A_\ell(t)+\phi_\ell(1)-\varrho-\tfrac 12 (A_\ell(t)-|\varrho|-R) -R \geqslant  \phi_\ell(1)+\tfrac 12 (A_\ell(t)-|\varrho|-R) \\
&= \phi_\ell(1)+ A_\ell(b_0)\geqslant \phi_\ell(1)+ A_\ell(b)=\phi_\ell(b).
\end{align*} In a similar way, one may prove \eqref{kernel estimate 2}. Hence, for $\gamma\geqslant 0$ (i.e. for $\delta\geqslant (\ell+1)^2$) by \eqref{kernel estimate 1} and \eqref{kernel estimate 2} we obtain
\begin{align}
\left((\phi_\ell(t)+\phi_\ell(b))^2-(\varrho-\eta)^2\right)^{-\gamma} \geqslant 2^{-2\gamma} \left(\phi_\ell^2(t)-\varrho^2\right)^{-\gamma}  \label{kernel estimate 3}
\end{align}  for $b\in [1,b_0]$ and $\eta\in [-(R+A_\ell(b)),R+A_\ell(b)]$. On the other hand, since for $b\in [1,b_0]$ and $\eta\in [-(R+A_\ell(b)),R+A_\ell(b)]$
\begin{align*}
\phi_\ell(t)+\phi_\ell(b)-\varrho +\eta\geqslant  \phi_\ell(t)+\phi_\ell(b)-\varrho -R-A_\ell(b)= \phi_\ell(t)-\varrho+\phi_\ell(1)-R, \\
\phi_\ell(t)+\phi_\ell(b)+\varrho -\eta \geqslant  \phi_\ell(t)+\phi_\ell(b)+\varrho -R-A_\ell(b)= \phi_\ell(t)+\varrho+\phi_\ell(1)-R,
\end{align*}  when $\gamma\leqslant 0$ (i.e. for $\delta\in [0,(\ell+1)^2$) we get
\begin{align}
\left((\phi_\ell(t)+\phi_\ell(b))^2-(\varrho-\eta)^2\right)^{-\gamma} \geqslant \left((\phi_\ell(t)+\phi_\ell(1)-R)^2-\varrho^2\right)^{-\gamma}.  \label{kernel estimate 4}
\end{align} Setting 
\begin{align*}
R_1\doteq \begin{cases} 0 & \mbox{if}\ \ \gamma\geqslant 0, \\ \phi_\ell(1)-R & \mbox{if}  \ \  \gamma< 0, \end{cases}
\end{align*} we can summarize \eqref{kernel estimate 3} and \eqref{kernel estimate 4} through the following estimate
\begin{align}
\left((\phi_\ell(t)+\phi_\ell(b))^2-(\varrho-\eta)^2\right)^{-\gamma}  \gtrsim  \left((\phi_\ell(t)+R_1)^2-\varrho^2\right)^{-\gamma} \label{kernel estimate 5}
\end{align} for $b\in [1,b_0]$ and $\eta\in [-(R+A_\ell(b)),R+A_\ell(b)]$. Combining \eqref{Radon transf ineq 2} and \eqref{kernel estimate 5}, we arrive at
\begin{align}
\mathcal{R}[u](t,\varrho) & \gtrsim t^{-\frac{\mu}{2}+\frac{1-\sqrt{\delta}}{2}} \left((\phi_\ell(t)+R_1)^2-\varrho^2\right)^{-\gamma}  \int_{1}^{b_0}  b^{\frac{\mu}{2}+\frac{1-\sqrt{\delta}}{2}} \int_{-(R+A_\ell(b))}^{R+A_\ell(b)} \mathcal{R}[|u|^p](b,\eta)  \mathrm{d}\eta \, \mathrm{d}b \notag  \\
& \gtrsim t^{-\frac{\mu}{2}+\frac{1-\sqrt{\delta}}{2}} \left((\phi_\ell(t)+R_1)^2-\varrho^2\right)^{-\gamma}  \int_{1}^{b_0}  b^{\frac{\mu}{2}+\frac{1-\sqrt{\delta}}{2}} \int_{\mathbb{R}} \mathcal{R}[|u|^p](b,\eta)  \mathrm{d}\eta \, \mathrm{d}b \notag \\
& = t^{-\frac{\mu}{2}+\frac{1-\sqrt{\delta}}{2}} \left((\phi_\ell(t)+R_1)^2-\varrho^2\right)^{-\gamma}  \int_{1}^{b_0}  b^{\frac{\mu}{2}+\frac{1-\sqrt{\delta}}{2}} \int_{\mathbb{R}^n} |u(b,x)|^p  \mathrm{d}x \, \mathrm{d}b \label{Radon transf fund ineq 1}
\end{align}
for $|\varrho|\leqslant A_\ell(t)-R$, where in the second step we used the support condition for $\mathcal{R}[|u|^p](s,\cdot)$.

Next, we follow the approach from \cite[Section 5]{PalRei18} and in \cite[Section C]{PalTak22} to obtain a lower bound estimate for $\| u(t,\cdot)\|^p_{L^p(\mathbb{R}^n)}$ in which the Radon transform of $u(t,\cdot)$ appears on the right-hand side. We introduce the operator 
\begin{align*}
\mathcal{T}_t(h)(\tau) \doteq |A_\ell(t)+R-\tau|^{-\frac{n-1}{2}}\int_\tau^ {A_\ell(t)+R}h(r)|r-\tau|^{\frac{n-3}{2}}\mathrm{d}r
\end{align*}  for any $t\geqslant 1$, $\tau\in\mathbb{R}$ and any $h\in L^p(\mathbb{R})$. The operator $\mathcal{T}_t$ is  a generalization of the operator introduced in \cite[Eq.(2.16)]{YZ06} in the critical case for the classical semilinear wave equation.  In \cite[Section 5]{PalRei18} it is proved that $\{\mathcal{T}_t\}_{t\geqslant 1}$ is a uniformly bounded  family of operators in $\mathcal{L}(L^p(\mathbb{R})\to L^p(\mathbb{R}))$ for any $n\geqslant 2$ and $p>1$. In particular, for $n\geqslant 3$ one can use the $L^p$-boundedness of the Hardy-Littlewood maximal function, while for $n=2$ a more careful analysis with weak Lebesgue spaces is necessary (cf. \cite[pages 1209-1210]{PalRei18}).
Let us introduce the  function
\begin{align*}
h(t,r)\doteq \begin{cases}
|u(t,r)|r^{\frac{n-1}{p}} & \mbox{if} \ r\geqslant 0, \\
0  & \mbox{if} \ r< 0. \end{cases}
\end{align*} Thanks to the uniform boundedness in $L^p(\mathbb{R})$ of $\{\mathcal{T}_t\}_{t\geqslant 1}$, it holds $\|\mathcal{T}_t(h)(t,\cdot)\|_{L^p(\mathbb{R})}\lesssim \|h(t,\cdot)\|_{L^p(\mathbb{R})}$ for any $t\in [1,T)$. Then,
\begin{align}
\int_{\mathbb{R}^n}|u(t,x)|^p \mathrm{d}x &= \omega_n\int^{\infty}_0|u(t,r)|^p r^{n-1}\mathrm{d}r=  \omega_n \|h(t,\cdot)\|_{L^p(\mathbb{R})}^p  \gtrsim  \|\mathcal{T}_t(h)(t,\cdot)\|_{L^p(\mathbb{R})} = \int_{\mathbb{R}}|\mathcal{T}_t(h)(t,\varrho)|^p \mathrm{d}\varrho \notag \\
&=\int_{\mathbb{R}} |A_\ell(t)+R-\varrho|^{-\frac{n-1}{2}p}\left|\int_\varrho^ {A_\ell(t)+R}|u(t,r)| r^{\frac{n-1}{p}}|r-\varrho|^{\frac{n-3}{2}}\mathrm{d}r \right|^p\mathrm{d}\varrho \notag \\
& \geqslant \int_{0}^{A_\ell(t)+R} (A_\ell(t)+R-\varrho)^{-\frac{n-1}{2}p}\left(\int_\varrho^ {A_\ell(t)+R}|u(t,r)| r^{\frac{n-1}{p}}(r-\varrho)^{\frac{n-3}{2}}\mathrm{d}r \right)^p\mathrm{d}\varrho \label{Rad inter 1}
\end{align}
Being the Radon transform a monotone operator, we have
\begin{align}
0\leqslant \mathcal{R}[u](t,\varrho)\leqslant \mathcal{R}[|u|](t,\varrho) &= \omega_{n-1}\int_{|\varrho|}^{A_\ell(t)+R}|u(t,r)|(r^2-\varrho^2)^{\frac{n-3}{2}} r \mathrm{d}r \notag 
\\ & \lesssim \int_{|\varrho|}^{A_\ell(t)+R}|u(t,r)|(r-|\varrho|)^{\frac{n-3}{2}} r^{\frac{n-1}{2}} \mathrm{d}r. \label{Rad inter 2}
\end{align} Furthermore, for any $\varrho\in [0,A_\ell(t)+R]$ and $r\in [\varrho,A_\ell(t)+R]$ 
\begin{align*}
r^{\frac{n-1}{p}}\geqslant (A_\ell(t)+R)^{-(n-1)[1-\frac{p}{2}]_{-}}r^{\frac{n-1}{2}} \varrho^{(n-1)[1-\frac{p}{2}]_{+}},
\end{align*} where $[1-\frac{p}{2}]_{\pm}$ denote the positive and the negative part of $1-\frac{p}{2}$, respectively. Thanks to the last estimate, from \eqref{Rad inter 1} and \eqref{Rad inter 2} we obtain
\begin{align}
\| u(t,\cdot)\|^p_{L^p(\mathbb{R}^n)} & \gtrsim  (A_\ell(t)+R)^{-(n-1)[1-\frac{p}{2}]_{-}} \int_0^{A_\ell(t)+R} \frac{\varrho^{(n-1)[1-\frac{p}{2}]_{+}}}{(A_\ell(t)+R-\varrho)^{\frac{n-1}{2}p}}\left(\int_{\varrho}^{A_\ell(t)+R}\!|u(t,r)|(r-\varrho)^{\frac{n-3}{2}} r^{\frac{n-1}{2}} \mathrm{d}r\! \right)^p \!\mathrm{d}\varrho \notag \\
& \gtrsim (A_\ell(t)+R)^{-(n-1)[1-\frac{p}{2}]_{-}}\int_0^{A_\ell(t)+R}\frac{\varrho^{(n-1)[1-\frac{p}{2}]_{+}}}{(A_\ell(t)+R-\varrho)^{\frac{n-1}{2}p}}\left(\mathcal{R}[u](t,\varrho)\right)^p \mathrm{d}\varrho. \label{Radon transf fund ineq 2}
\end{align} 
Combining \eqref{Radon transf fund ineq 1} and \eqref{Radon transf fund ineq 2}, we derive the following iteration frame for $\|u(t,\cdot)\|_{L^p(\mathbb{R}^n)}^p$
\begin{align}
\| u(t,\cdot)\|^p_{L^p(\mathbb{R}^n)}  \geqslant  \frac{K \, t^{-\frac{\mu p}{2}+\frac{1-\sqrt{\delta}}{2}p}}{(A_\ell(t)+R)^{(n-1)[1-\frac{p}{2}]_{-}}} & \int_0^{A_\ell(t)-R} \varrho^{(n-1)[1-\frac{p}{2}]_{+}} \frac{\left((\phi_\ell(t)+R_1)^2-\varrho^2\right)^{-\gamma p}}{(A_\ell(t)+R-\varrho)^{\frac{n-1}{2}p}} \notag \\
& \quad \times  \left( \int_{1}^{A_\ell^{-1}\left(\frac{1}{2}(A_\ell(t)-\varrho-R)\right)}  b^{\frac{\mu}{2}+\frac{1-\sqrt{\delta}}{2}}  \|u(b,\cdot)\|^p_{L^p(\mathbb{R}^n)}\mathrm{d}b\right)^p \mathrm{d}\varrho,
\label{iteration frame ||u(t)||^p Lp }
\end{align} for $t\geqslant A_\ell^{-1}(R)$ and for a suitable positive constant $K=K(\mu,\nu^2,\ell,n,p)$.

\section{Iteration argument for $\|u(t,\cdot)\|^p_{L^p(\mathbb{R}^n)}$} \label{Section iteration  argument}

In this section we are going to derive a sequence of lower bound estimates for $\|u(t,\cdot)\|_{L^p(\mathbb{R}^n)}^p$ that improve \eqref{1st low bound ||u||^p} by including a factor with logarithmic growth. In order to achieve this goal, we employ the iteration frame \eqref{iteration frame ||u(t)||^p Lp } and an inductive argument. For the sake of readability, we split the proof of the base case, the inductive step and the improvement of the lower bound estimates for $U(t)$ into three different subsections.

\subsection{Base case}\label{Subsection base case}

We begin our inductive argument, by plugging \eqref{1st low bound ||u||^p} into \eqref{iteration frame ||u(t)||^p Lp }. Thus, for $t\geqslant A_\ell^{-1}(R+2A_\ell(T_2))$ we have
\begin{align}
\| u(t,\cdot)\|^p_{L^p(\mathbb{R}^n)} & \geqslant   \frac{K \widetilde{C}^p \varepsilon^{p^2} \, t^{-\frac{\mu p}{2}+\frac{1-\sqrt{\delta}}{2}p}}{(A_\ell(t)+R)^{(n-1)[1-\frac{p}{2}]_{-}}} \int_0^{A_\ell(t)-R-2A_\ell(T_2)} \varrho^{(n-1)[1-\frac{p}{2}]_{+}} \frac{\left((\phi_\ell(t)+R_1)^2-\varrho^2\right)^{-\gamma p}}{(A_\ell(t)+R-\varrho)^{\frac{n-1}{2}p}}   \left( I_0(t,\varrho)\right)^p \mathrm{d}\varrho, \label{low bound ||u||^p base case 1}
\end{align} where
\begin{align*}
I_0(t,\varrho)\doteq \int_{T_2}^{A_\ell^{-1}\left(\frac{1}{2}(A_\ell(t)-\varrho-R)\right)}  b^{\frac{\mu}{2}+\frac{1-\sqrt{\delta}}{2}-\big(\frac{n-1}{2}(\ell+1)+\frac{\ell+\mu}{2}\big)p+(n-1)(\ell+1)} \mathrm{d}b.
\end{align*} Let us estimate $I_0(t,\varrho)$ from below, by shrinking the domain of integration with respect to $\varrho$ in \eqref{low bound ||u||^p base case 1}. Since the inequality $2(r_2+1)=\mu+1-\sqrt{\delta}\geqslant 0$ is true for any $\mu,\nu^2$ satisfying $\delta\geqslant 0$, we find
\begin{align*}
I_0(t,\varrho) \geqslant & \left(A_\ell^{-1}\left(\tfrac{1}{2}(A_\ell(t)-\varrho-R)\right)\right)^{-\big(\frac{n-1}{2}(\ell+1)+\frac{\ell+\mu}{2}\big)p} \int_{T_2}^{A_\ell^{-1}\left(\frac{1}{2}(A_\ell(t)-\varrho-R)\right)}  b^{r_2+1+(n-1)(\ell+1)} \mathrm{d}b \\
 =& \left(A_\ell^{-1}\!\left(\tfrac{1}{2}(A_\ell(t)-\varrho-R)\right)\right)^
 {-\big(\frac{n-1}{2}(\ell+1)+\frac{\ell+\mu}{2}\big)p}\frac{ \left(A_\ell^{-1}\!\left(\tfrac{1}{2}(A_\ell(t)-\varrho-R)\right)\right)^{r_2+2+(n-1)(\ell+1)}-T_2^{r_2+2+(n-1)(\ell+1)}}{r_2+2+(n-1)(\ell+1)}.
\end{align*} Since $A_\ell^{-1}(\sigma)=((\ell+1)\sigma+1)^{\frac{1}{\ell+1}}$, we have
\begin{align*}
A_\ell^{-1}\left(\tfrac{1}{2}(A_\ell(t)-\varrho-R)\right) &=\left(\tfrac{\ell+1}{2}\right)^{\frac{1}{\ell+1}} \left(A_\ell(t)-\varrho-R+2\phi_\ell(1)\right)^{\frac{1}{\ell+1}}.
\end{align*} Let us introduce a parameter $\alpha_0>2(T_2^{\ell+1}-1)$.  Then, $A_\ell(t)-R-\alpha_0 \phi_\ell(1)\leqslant A_\ell(t)-R-2A_\ell(T_2)$, so for $\varrho \in [0,A_\ell(t)-R-\alpha_0 \phi_\ell(1)]$ we have 
\begin{align*}
I_0(t,\varrho) &\geqslant \widehat{B} \left(A_\ell(t)-\varrho-R+2\phi_\ell(1)\right)^{-\big(\frac{n-1}{2}
+\frac{\ell+\mu}{2(\ell+1)}\big)p+\frac{r_2+2}{\ell+1}+n-1},
\end{align*} where the multiplicative constant is given by
\begin{align*}
\widehat{B}\doteq \frac{\left(\tfrac{\ell+1}{2}\right)^{-\big(\frac{n-1}{2}+\frac{\ell+\mu}{2(\ell+1)}\big)p+\frac{r_2+1}{\ell+1}+n-1} }{r_2+2+(n-1)(\ell+1)}\left[1-\left(\tfrac{2T_2^{\ell+1}}{\alpha_0+2}\right)^{\frac{r_2+2}{\ell+1}+n-1}\right].
\end{align*} Hence, combining the previous estimate for $I_0(t,\varrho)$ with \eqref{low bound ||u||^p base case 1}, for $t\geqslant A^{-1}_\ell(R+\alpha_0\phi_\ell(1))$ we obtain
\begin{align}
\| u(t,\cdot)\|^p_{L^p(\mathbb{R}^n)}  \geqslant  K \widetilde{C}^p \widehat{B}^p \varepsilon^{p^2}  \frac{t^{-\frac{\mu p}{2}+\frac{1-\sqrt{\delta}}{2}p}}{(A_\ell(t)+R)^{(n-1)[1-\frac{p}{2}]_{-}}} J_0(t), \label{low bound ||u||^p base case 2}
\end{align} where
\begin{align*}
J_0(t)\doteq  \int_0^{A_\ell(t)-R-\alpha_0 \phi_\ell(1)} & \varrho^{(n-1)[1-\frac{p}{2}]_{+}} \frac{\left((\phi_\ell(t)+R_1)^2-\varrho^2\right)^{-\gamma p}}{(A_\ell(t)+R-\varrho)^{\frac{n-1}{2}p}}  \\ &\times \left(A_\ell(t)-\varrho-R+2\phi_\ell(1)\right)^{-\big(\frac{n-1}{2}+\frac{\ell+\mu}{2(\ell+1)}\big)p^2+\big(\frac{r_2+2}{\ell+1}+n-1\big)p} \mathrm{d}\varrho.
\end{align*} The next step is to estimate from below the factor $\left((\phi_\ell(t)+R_1)^2-\varrho^2\right)^{-\gamma p}$ in $J_0(t)$: we consider separately the case $\gamma\geqslant 0$ from the case $\gamma <0$. If $\gamma\geqslant 0$, then for $\varrho \in [0,A_\ell(t)-R-\alpha_0 \phi_\ell(1)]$
\begin{align*}
\phi_\ell(t)+R_1-\varrho & =A_\ell(t)-\varrho+\phi_\ell(1), \\
\phi_\ell(t)+R_1+\varrho & =\phi_\ell(t)+\varrho\leqslant 2\phi_\ell(t)-R-(\alpha_0+1)\leqslant 2\phi_\ell(t), 
\end{align*} and, consequently,
\begin{align}
\left((\phi_\ell(t)+R_1)^2-\varrho^2\right)^{-\gamma p} & \geqslant \left(\tfrac{\ell+1}{2}\right)^{\gamma p} t^{-(\ell+1)\gamma p} \left(A_\ell(t)-\varrho+\phi_\ell(1)\right)^{-\gamma p} \notag \\
 & = \left(\tfrac{\ell+1}{2}\right)^{\gamma p} t^{-\frac{\ell p}{2}-\frac{1-\sqrt{\delta}}{2}p} \left(A_\ell(t)-\varrho+\phi_\ell(1)\right)^{-\gamma p}. \label{factor with R1 gamma>=0}
\end{align} On the other hand, for $\gamma<0$, we introduce another parameter $a_0\geqslant 2$ and  for $t\geqslant A_\ell^{-1}(a_0R+\alpha_0 \phi_\ell(1))$ we consider  $\varrho \in [0,A_\ell(t)-a_0R-\alpha_0 \phi_\ell(1)]$
\begin{align*}
\phi_\ell(t)+R_1-\varrho & =A_\ell(t)-\varrho-R+2\phi_\ell(1), \\
\phi_\ell(t)+R_1+\varrho & =\phi_\ell(t)+\varrho-R+\phi_\ell(1) \geqslant \phi_\ell(t)-R+\phi_\ell(1) \\ & \geqslant \begin{cases} \phi_\ell(t) & \mbox{if} \ R\leqslant \phi_\ell(1), \\ \tfrac 12 \phi_\ell(t) & \mbox{if} \ R> \phi_\ell(1) \ \mbox{and} \ t\geqslant A_\ell^{-1}(2R-3\phi_\ell(1)). \end{cases} 
\end{align*}  Notice that, since we are assuming $a_0\geqslant 2$, the condition $t\geqslant A_\ell^{-1}(2R-3\phi_\ell(1))$ is always fulfilled for the case that we are considering. Therefore, when $\gamma<0$ for $t\geqslant A_\ell^{-1}(a_0R+\alpha_0 \phi_\ell(1))$ and  $\varrho \in [0,A_\ell(t)-a_0R-\alpha_0 \phi_\ell(1)]$, it holds
\begin{align}
\left((\phi_\ell(t)+R_1)^2-\varrho^2\right)^{-\gamma p} & \geqslant \left(2(\ell+1)\right)^{\gamma p} t^{-\frac{\ell p}{2}-\frac{1-\sqrt{\delta}}{2}p} \left(A_\ell(t)-\varrho-R+2\phi_\ell(1)\right)^{-\gamma p}. \label{factor with R1 gamma<0}
\end{align}
Combining \eqref{factor with R1 gamma>=0} and \eqref{factor with R1 gamma<0} with the upper bound estimates
\begin{align*}
 A_\ell(t)-\varrho+\phi_\ell(1) &\leqslant \tfrac{a_0}{a_0-1} \left(A_\ell(t)-\varrho-R+2\phi_\ell(1)\right), \\
 A_\ell(t)-\varrho+R &\leqslant \tfrac{a_0+1}{a_0-1} \left(A_\ell(t)-\varrho-R+2\phi_\ell(1)\right),
\end{align*}
for $\varrho\in [0,A_\ell(t)-a_0R-\alpha_0\phi_\ell(1)]$ and $t\geqslant A_\ell^{-1}(a_0R+\alpha_0 \phi_\ell(1))$ we arrive at the following lower bound for $J_0(t)$
\begin{align*}
J_0(t)\geqslant \widetilde{B} t^{-\frac{\ell p}{2}-\frac{1-\sqrt{\delta}}{2}p}  \int_0^{A_\ell(t)-a_0R-\tfrac{\alpha_0}{\ell+1}} \varrho^{(n-1)[1-\frac{p}{2}]_{+}}  \left(A_\ell(t)-\varrho-R+2\phi_\ell(1)\right)^{-\big(\frac{n-1}{2}+\frac{\ell+\mu}{2(\ell+1)}\big)p^2+\big(\frac{r_2+2}{\ell+1}+\frac{n-1}{2}-\gamma\big)p} \mathrm{d}\varrho,
\end{align*} where the multiplicative constant is defined as
\begin{align} \label{def widetildeB}
 \widetilde{B}\doteq \begin{cases}
 \left(\tfrac{\ell+1}{2}\right)^{\gamma p} \left(\tfrac{a_0}{a_0-1}\right)^{-\gamma p}  \left(\tfrac{a_0+1}{a_0-1}\right)^{-\frac{n-1}{2} p} & \mbox{if} \ \gamma\geqslant 0, \\
  \left(2(\ell+1)\right)^{\gamma p}   \left(\tfrac{a_0+1}{a_0-1}\right)^{-\frac{n-1}{2} p} & \mbox{if} \ \gamma<0.
 \end{cases}
\end{align} 
Let us remark that, since $p=p_{\mathrm{Str}}\big(n+\tfrac{\mu}{\ell+1},\ell\big)$, by using \eqref{equation for p crit}, we can rewrite the exponent in the last integral as
\begin{align*}
-\left(\tfrac{n-1}{2}+\tfrac{\ell+\mu}{2(\ell+1)}\right)p^2+\left(\tfrac{r_2+2}{\ell+1}+\tfrac{n-1}{2}-\gamma\right)p &= -\left(\tfrac{n-1}{2}+\tfrac{\ell+\mu}{2(\ell+1)}\right)p^2+\left(\tfrac{n+1}{2}+\tfrac{\mu-3\ell}{2(\ell+1)}\right)p =-1.
\end{align*} Let us fix a constant $\theta \in (0,1)$. We keep going with our estimate from below for $J_0(t)$ by shrinking further the domain of integration. For $\varrho\in [0,A_\ell(t)-a_0R-\alpha_0\phi_\ell(1)]$ and $t\geqslant A_\ell^{-1}(a_0R+\alpha_0 \phi_\ell(1))$ we have
\begin{align*}
J_0(t) & \geqslant \widetilde{B} t^{-\frac{\ell p}{2}-\frac{1-\sqrt{\delta}}{2}p}   \int_0^{A_\ell(t)-a_0R-\alpha_0\phi_\ell(1)} \varrho^{(n-1)[1-\frac{p}{2}]_{+}}  \left(A_\ell(t)-\varrho-R+2\phi_\ell(1)\right)^{-1} \mathrm{d}\varrho \\
& \geqslant \widetilde{B} t^{-\frac{\ell p}{2}-\frac{1-\sqrt{\delta}}{2}p}   \int_{\theta(A_\ell(t)-a_0R-\alpha_0\phi_\ell(1))}^{A_\ell(t)-a_0R-\alpha_0\phi_\ell(1)}  \varrho^{(n-1)[1-\frac{p}{2}]_{+}}  \left(A_\ell(t)-\varrho-R+2\phi_\ell(1)\right)^{-1} \mathrm{d}\varrho \\
& \geqslant \widetilde{B} \theta^{(n-1)[1-\frac{p}{2}]_{+}} t^{-\frac{\ell p}{2}-\frac{1-\sqrt{\delta}}{2}p}  (A_\ell(t)-a_0R-\alpha_0\phi_\ell(1))^{(n-1)[1-\frac{p}{2}]_{+}}  \ln\left(\frac{A_\ell(t)+\frac{\theta a_0-1}{1-\theta}R+\frac{\theta \alpha_0+2}{1-\theta}\phi_\ell(1)}{\frac{a_0-1}{1-\theta}R+\frac{ \alpha_0+2}{1-\theta}\phi_\ell(1)}\right).
\end{align*}
By using the previous lower bound for $J_0(t)$ in \eqref{low bound ||u||^p base case 2}, for $t\geqslant A_\ell^{-1}(a_0R+\alpha_0 \phi_\ell(1))$ we have
\begin{align}\label{1st low bound ||u||^p with log}
\|u(t,\cdot)\|_{L^p(\mathbb{R}^n)}^p   \geqslant B_0\varepsilon^{p^2} &  t^{-\frac{\ell+\mu }{2}p}\frac{(A_\ell(t)-a_0R-\alpha_0\phi_\ell(1))^{(n-1)[1-\frac{p}{2}]_{+}}}{(A_\ell(t)+R)^{(n-1)[1-\frac{p}{2}]_{-}}} \ln\left(\frac{A_\ell(t)+\frac{\theta a_0-1}{1-\theta}R+\frac{\theta \alpha_0+2}{1-\theta}\phi_\ell(1)}{\frac{a_0-1}{1-\theta}R+\frac{ \alpha_0+2}{1-\theta}\phi_\ell(1)}\right),
\end{align} where $B_0\doteq  K \widetilde{C}^p  \widehat{B}^p \widetilde{B} \theta^{(n-1)[1-\frac{p}{2}]_{+}} $. 

\subsection{Inductive step}

In Subsection \ref{Subsection base case} we determined the first lower bound estimate \eqref{1st low bound ||u||^p with log} for $\|u(t,\cdot)\|_{L^p(\mathbb{R}^n)}^p $ which improves \eqref{1st low bound ||u||^p} thanks to the presence of a logarithmic factor. The goal of the present subsection is to derive a sequence of lower bound estimates for $\|u(t,\cdot)\|_{L^p(\mathbb{R}^n)}^p$  with additional logarithmic factors.

\begin{lemma} Let us assume that $\theta\in (\frac{1}{2},1)$ and that the real parameters $a_0,\alpha_0$ satisfy 
\begin{align}\label{a0, alpha0 conditions}
a_0\geqslant\max\left\{2,\frac{1}{2\theta -1}\right\} \ \mbox{and} \ \  \alpha_0>2(T_2^{\ell+1}-1).
\end{align} We consider the sequences $\{a_j\}_{j\in\mathbb{N}}, \{\alpha_j\}_{j\in\mathbb{N}}$ whose terms for any $j\in\mathbb{N}, j\geqslant 1$ are given by
\begin{align}
a_j &\doteq (a_0-1)\left(\frac{4}{1-\theta}\right)^j+1, \label{def aj} \\
\alpha_j &\doteq (\alpha_0+2)\left(\frac{4}{1-\theta}\right)^j-2. \label{def alphaj}
\end{align}
Then, for any $j\in\mathbb{N}$ and any $t\geqslant A_\ell^{-1}(a_jR+\alpha_j \phi_\ell(1))$ we have
\begin{align} \label{(j+1)-th low bound ||u||^p with log}
\|u(t,\cdot)\|_{L^p(\mathbb{R}^n)}^p   \geqslant B_j\varepsilon^{p^{j+2}}   t^{-\frac{\ell+\mu }{2}p}\frac{(A_\ell(t)-a_jR-\alpha_j\phi_\ell(1))^{(n-1)[1-\frac{p}{2}]_{+}}}{(A_\ell(t)+R)^{(n-1)[1-\frac{p}{2}]_{-}}} \left(\!\ln\left(\frac{A_\ell(t)+\frac{\theta a_j-1}{1-\theta}R+\frac{\theta \alpha_j+2}{1-\theta}\phi_\ell(1)}{\frac{a_j-1}{1-\theta}R+\frac{ \alpha_j+2}{1-\theta}\phi_\ell(1)}\right)\!\right)^{\frac{p^{j+1}-1}{p-1}},
\end{align} 
where  $\{B_j\}_{j\in\mathbb{N}}$ is a sequence of positive reals whose first term $B_0=B_0(n,\ell,\mu,\nu^2,p,R,u_0,u_1,a_0,\alpha_0,\theta)$ is the constant on the right hand side of \eqref{1st low bound ||u||^p with log} and the other terms fulfill the recursive relation
\begin{align}\label{recursive relation Bj}
B_{j+1} \doteq D p^{-(j+1)} B_j^p
\end{align} 
for any $j\in\mathbb{N}$ and for a suitable positive constant $D$ depending on $n,\ell,\mu,\nu^2,p,R,a_0,\alpha_0,\theta$.
\end{lemma}

\begin{proof}
We have already proved \eqref{(j+1)-th low bound ||u||^p with log} for $j=0$, namely, \eqref{1st low bound ||u||^p with log}. In order to prove \eqref{(j+1)-th low bound ||u||^p with log} by using an inductive argument, we check the inductive step. Let us assume that \eqref{(j+1)-th low bound ||u||^p with log} holds true for some $j\geqslant 0$. Obviously, our target is to show \eqref{(j+1)-th low bound ||u||^p with log} for $j+1$, specifying the recursive relations among $a_{j+1},\alpha_{j+1}, B_{j+1}$ and $a_{j},\alpha_{j}, B_{j}$. We begin by plugging \eqref{(j+1)-th low bound ||u||^p with log} for $j$ in \eqref{iteration frame ||u(t)||^p Lp }.  Then, for $t\geqslant A_\ell^{-1}((2a_j+1)R+2\alpha_j \phi_\ell(1))$ we obtain
\begin{align}
\| u(t,\cdot)\|^p_{L^p(\mathbb{R}^n)}  & \geqslant  KB_j^p \varepsilon^{p^{j+3}} t^{-\frac{\mu p}{2}+\frac{1-\sqrt{\delta}}{2}p} (A_\ell(t)+R)^{-(n-1)[1-\frac{p}{2}]_{-}} J_{j+1}(t), \label{l.b.e. |u(t)|^p_L^p begin inductive step}
\end{align} where
\begin{align*}
 J_{j+1}(t) & \doteq \int_0^{A_\ell(t)-(2a_j+1)R-2\alpha_j \phi_\ell(1)} \varrho^{(n-1)[1-\frac{p}{2}]_{+}} \frac{\left((\phi_\ell(t)+R_1)^2-\varrho^2\right)^{-\gamma p}}{(A_\ell(t)+R-\varrho)^{\frac{n-1}{2}p}} \left(I_{j+1}(t,\varrho)\right)^p \mathrm{d}\varrho, 
 \\
I_{j+1}(t,\varrho) & \doteq  \int_{A_\ell^{-1}(a_jR+\alpha_j \phi_\ell(1))}^{A_\ell^{-1}\left(\frac{1}{2}(A_\ell(t)-\varrho-R)\right)}  b^{\frac{\mu}{2}+\frac{1-\sqrt{\delta}}{2}-\frac{\ell+\mu}{2}p} \frac{(A_\ell(b)-a_jR-\alpha_j\phi_\ell(1))^{(n-1)[1-\frac{p}{2}]_{+}}}{(A_\ell(b)+R)^{(n-1)[1-\frac{p}{2}]_{-}}} \notag\\
& \qquad \qquad \qquad \qquad \qquad\times \left(\ln\left(\frac{A_\ell(b)+\frac{\theta a_j-1}{1-\theta}R+\frac{\theta \alpha_j+2}{1-\theta}\phi_\ell(1)}{\frac{a_j-1}{1-\theta}R+\frac{ \alpha_j+2}{1-\theta}\phi_\ell(1)}\right)\right)^{\frac{p^{j+1}-1}{p-1}}\mathrm{d}b. 
\end{align*} 
\begin{remark} We stress that, when we use the lower bound estimate \eqref{(j+1)-th low bound ||u||^p with log} in \eqref{iteration frame ||u(t)||^p Lp }, in order to avoid empty domain of integration in the $b$-integral we must shrink the domain of integration with respect to $\varrho$ from $[0,A_\ell(t)-R]$ to $[0,A_\ell(t)-(2a_j+1)R-2\alpha_j \phi_\ell(1)]$ and require  $t\geqslant A_\ell^{-1}((2a_j+1)R+2\alpha_j \phi_\ell(1))$ so that this domain of integration in $J_{j+1}(t)$ is not empty. Hereafter, we make this kind of considerations implicitly every time that we shrink an interval of integration.
\end{remark}
Let us start with the estimate of the integral $I_{j+1}(t,\varrho)$. Clearly,  $A_\ell(b)-a_j R-\alpha_j \phi_\ell(1)\geqslant \frac{1}{2}\phi_\ell(b)$ if and only if $b\geqslant A_\ell^{-1}(2a_j R+(2\alpha_j+1)\phi_\ell(1))$ and $A_\ell(b)+R\leqslant 2\phi_\ell(b)$ if and only if $b\geqslant A_\ell^{-1}(R-2\phi_\ell(1))$. From \eqref{def aj} and \eqref{def alphaj}  it follows that $2a_j>1$ and $2\alpha_j+1>-2$. Therefore, for $t\geqslant A_\ell^{-1}((4a_j+1)R+2(2\alpha_j+1)\phi_\ell(1))$ and $\varrho\in[0,A_\ell(t)-(4a_j+1)R-2(2\alpha_j+1)\phi_\ell(1)]$ we can shrink the domain of integration in $I_{j+1}(t,\varrho)$ from  $\left[A_\ell^{-1}(a_jR+\alpha_j \phi_\ell(1)),A_\ell^{-1}\left(\frac{1}{2}(A_\ell(t)-\varrho-R)\right)\right]$ 
to  $\left[A_\ell^{-1}(2a_jR+(2\alpha_j+1) \phi_\ell(1)),A_\ell^{-1}\left(\frac{1}{2}(A_\ell(t)-\varrho-R)\right)\right]$. 

Consequently, for $t\geqslant A_\ell^{-1}((4a_j+1)R+2(2\alpha_j+1)\phi_\ell(1))$ and $\varrho\in[0,A_\ell(t)-(4a_j+1)R-2(2\alpha_j+1)\phi_\ell(1)]$ we find
\begin{align}
I_{j+1}(t,\varrho) & \geqslant  C_1 \int_{A_\ell^{-1}(2a_jR+(2\alpha_j+1) \phi_\ell(1))}^{A_\ell^{-1}\left(\frac{1}{2}(A_\ell(t)-\varrho-R)\right)}  b^{r_2+1-\frac{\ell+\mu}{2}p+(n-1)(\ell+1)(1-\frac{p}{2})} \left(\!\ln\left(\frac{A_\ell(b)+\frac{\theta a_j-1}{1-\theta}R+\frac{\theta \alpha_j+2}{1-\theta}\phi_\ell(1)}{\frac{a_j-1}{1-\theta}R+\frac{ \alpha_j+2}{1-\theta}\phi_\ell(1)}\right)\!\right)^{\frac{p^{j+1}-1}{p-1}}\!\mathrm{d}b \notag\\
 & \geqslant C_1 \left(A_\ell^{-1}\left(\tfrac{1}{2}(A_\ell(t)-\varrho-R)\right)\right)^{-(\frac{n-1}{2}(\ell+1)+\frac{\ell+\mu}{2})p} \widetilde{I}_{j+1}(t,\varrho), \label{1st l.b.e. Ij+1(t,rho)}
\end{align} where $C_1\doteq 2^{-(n-1)|1-\frac{p}{2}|} (\ell+1)^{(n-1)(\frac{p}{2}-1)}$ and 
\begin{align*}
\widetilde{I}_{j+1}(t,\varrho) \doteq  \int_{A_\ell^{-1}(2a_jR+(2\alpha_j+1) \phi_\ell(1))}^{A_\ell^{-1}\left(\frac{1}{2}(A_\ell(t)-\varrho-R)\right)}  b^{r_2+1+(n-1)(\ell+1)} \left(\ln\left(\frac{A_\ell(b)+\frac{\theta a_j-1}{1-\theta}R+\frac{\theta \alpha_j+2}{1-\theta}\phi_\ell(1)}{\frac{a_j-1}{1-\theta}R+\frac{ \alpha_j+2}{1-\theta}\phi_\ell(1)}\right)\right)^{\frac{p^{j+1}-1}{p-1}}\mathrm{d}b.
\end{align*}
For $t\geqslant A_\ell^{-1}((8a_j+1)R+4(2\alpha_j+1)\phi_\ell(1))$ and $\varrho\in[0, A_\ell(t)-(8a_j+1)R-4(2\alpha_j+1)\phi_\ell(1))]$, we shrink further the region of integration in $\widetilde{I}_{j+1}(t,\varrho)$ as follows
\begin{align}
\widetilde{I}_{j+1}(t,\varrho)  \geqslant & \int_{A_\ell^{-1}\left(\frac{1}{4}(A_\ell(t)-\varrho-R)\right)}^{A_\ell^{-1}\left(\frac{1}{2}(A_\ell(t)-\varrho-R)\right)}  b^{r_2+1+(n-1)(\ell+1)} \left(\ln\left(\frac{A_\ell(b)+\frac{\theta a_j-1}{1-\theta}R+\frac{\theta \alpha_j+2}{1-\theta}\phi_\ell(1)}{\frac{a_j-1}{1-\theta}R+\frac{ \alpha_j+2}{1-\theta}\phi_\ell(1)}\right)\right)^{\frac{p^{j+1}-1}{p-1}}\mathrm{d}b \notag \\
 \geqslant& \left(\ln\left(\frac{A_\ell(t)-\varrho+\left(\frac{4(\theta a_j-1)}{1-\theta}-1\right)R+\frac{4(\theta \alpha_j+2)}{1-\theta}\phi_\ell(1)}{\frac{4(a_j-1)}{1-\theta}R+\frac{ 4(\alpha_j+2)}{1-\theta}\phi_\ell(1)}\right)\right)^{\frac{p^{j+1}-1}{p-1}} \int_{A_\ell^{-1}\left(\frac{1}{4}(A_\ell(t)-\varrho-R)\right)}^{A_\ell^{-1}\left(\frac{1}{2}(A_\ell(t)-\varrho-R)\right)}  b^{r_2+1+(n-1)(\ell+1)} \mathrm{d}b. \label{1st l.b.e. I tilde j+1(t,rho)}
\end{align}
Let us rewrite the integral in the right-hand side of the previous inequality in a more convenient way
\begin{align*}
& \int_{A_\ell^{-1}\left(\frac{1}{4}(A_\ell(t)-\varrho-R)\right)}^{A_\ell^{-1}\left(\frac{1}{2}(A_\ell(t)-\varrho-R)\right)}  b^{r_2+1+(n-1)(\ell+1)} \mathrm{d}b \\  & =  \frac{\left(A_\ell^{-1}\left(\tfrac{1}{2}(A_\ell(t)-\varrho-R)\right)\right)^{r_2+2+(n-1)(\ell+1)}}{r_2+2+(n-1)(\ell+1)}\left(1 -\left(\frac{A_\ell^{-1}\left(\tfrac{1}{4}(A_\ell(t)-\varrho-R)\right)}{A_\ell^{-1}\left(\tfrac{1}{2}(A_\ell(t)-\varrho-R)\right)}\right)^{r_2+2+(n-1)(\ell+1)}  \right).
\end{align*} Recalling that $A_\ell^{-1}(\sigma)=\left((\ell+1)\sigma+1\right)^{\frac{1}{\ell+1}}$, we have that
\begin{align*}
A_\ell^{-1}\left(\tfrac{1}{2}(A_\ell(t)-\varrho-R)\right)& = \left(\tfrac{\ell+1}{2}\right)^\frac{1}{\ell+1} \left(A_\ell(t)-\varrho-R+2\phi_\ell(1)\right)^\frac{1}{\ell+1},\\ 
A_\ell^{-1}\left(\tfrac{1}{4}(A_\ell(t)-\varrho-R)\right) &=  \left(\tfrac{\ell+1}{4}\right)^\frac{1}{\ell+1} \left(A_\ell(t)-\varrho-R+4\phi_\ell(1)\right)^\frac{1}{\ell+1} ,
\end{align*}
and hence, for $\varrho\in[0, A_\ell(t)-(8a_j+1)R-4(2\alpha_j+1)\phi_\ell(1))]$
\begin{align*}
& \int_{A_\ell^{-1}\left(\frac{1}{4}(A_\ell(t)-\varrho-R)\right)}^{A_\ell^{-1}\left(\frac{1}{2}(A_\ell(t)-\varrho-R)\right)}  b^{r_2+1+(n-1)(\ell+1)} \mathrm{d}b \\  & \qquad =C_2\left(1 -\left(\frac{1}{2}+\frac{\phi_\ell(1)}{A_\ell(t)-\varrho-R+2\phi_\ell(1)}\right)^{\frac{r_2+2}{\ell+1}+(n-1)}  \right) \big(A_\ell(t)-\varrho-R+2\phi_\ell(1)\big)^{\frac{r_2+2}{\ell+1} +n-1} \\  & \qquad  \geqslant
C_2\left(1 -\left(\frac{1}{2}+\frac{\phi_\ell(1)}{8a_j R+4(2\alpha_j+3)\phi_\ell(1)}\right)^{\frac{r_2+2}{\ell+1}+(n-1)}  \right) \big(A_\ell(t)-\varrho-R+2\phi_\ell(1)\big)^{\frac{r_2+2}{\ell+1} +n-1},
\end{align*} where $C_2\doteq (r_2+2+(n-1)(\ell+1))^{-1}\left(\tfrac{\ell+1}{2}\right)^{\frac{r_2+2}{\ell+1} +n-1}$. Now, pointing out that the sequences $\{a_j\}_{j\in\mathbb{N}}$, $\{\alpha_j\}_{j\in\mathbb{N}}$ are increasing, we get immediately that the sequence $\left\{1-(\frac{1}{2}+\frac{\phi_\ell(1)}{8a_j R+4(2\alpha_j+3)\phi_\ell(1)})\right\}_{j\in\mathbb{N}}$ is increasing as well. Thus, for $\varrho\in[0, A_\ell(t)-(8a_j+1)R-4(2\alpha_j+1)\phi_\ell(1))]$ we conclude that 
\begin{align*}
 \int_{A_\ell^{-1}\left(\frac{1}{4}(A_\ell(t)-\varrho-R)\right)}^{A_\ell^{-1}\left(\frac{1}{2}(A_\ell(t)-\varrho-R)\right)}  b^{r_2+1+(n-1)(\ell+1)} \mathrm{d}b  \geqslant
C_2 C_3 \big(A_\ell(t)-\varrho-R+2\phi_\ell(1)\big)^{\frac{r_2+2}{\ell+1} +n-1},
\end{align*} where $C_3\doteq 1 -\left(\frac{1}{2}+\frac{\phi_\ell(1)}{8a_0 R+4(2\alpha_0+3)\phi_\ell(1)}\right)^{\frac{r_2+2}{\ell+1}+(n-1)}  $. Combining the last inequality with  \eqref{1st l.b.e. I tilde j+1(t,rho)}, we obtain
\begin{align*}
\widetilde{I}_{j+1}(t,\varrho)  & \geqslant C_2 C_3 \!\left(\!\ln\!\left(\frac{A_\ell(t)-\varrho+\left(\frac{4(\theta a_j-1)}{1-\theta}-1\right)R+\frac{4(\theta \alpha_j+2)}{1-\theta}\phi_\ell(1)}{\frac{4(a_j-1)}{1-\theta}R+\frac{ 4(\alpha_j+2)}{1-\theta}\phi_\ell(1)}\right)\!\right)^{\frac{p^{j+1}-1}{p-1}} \!\big(A_\ell(t)-\varrho-R+2\phi_\ell(1)\big)^{\frac{r_2+2}{\ell+1} +n-1}.
\end{align*} The previous estimate together with \eqref{1st l.b.e. Ij+1(t,rho)} and \eqref{1st l.b.e. I tilde j+1(t,rho)} implies that
\begin{align*}
I_{j+1}(t,\varrho) &  \geqslant C_4 \big(A_\ell(t)-\varrho-R+2\phi_\ell(1)\big)^{-\frac{\ell+\mu}{2(\ell+1)}p+\frac{r_2+2}{\ell+1} +(n-1)(1-\frac{p}{2})} \\ &  \quad \times \left(\ln\left(\frac{A_\ell(t)-\varrho+\left(\frac{4(\theta a_j-1)}{1-\theta}-1\right)R+\frac{4(\theta \alpha_j+2)}{1-\theta}\phi_\ell(1)}{\frac{4(a_j-1)}{1-\theta}R+\frac{ 4(\alpha_j+2)}{1-\theta}\phi_\ell(1)}\right)\right)^{\frac{p^{j+1}-1}{p-1}} 
\end{align*}
for $t\geqslant A_\ell^{-1}((8a_j+1)R+4(2\alpha_j+1)\phi_\ell(1))$ and $\varrho\in[0, A_\ell(t)-(8a_j+1)R-4(2\alpha_j+1)\phi_\ell(1))]$, where the multiplicative constant is given by $C_4 \doteq C_1 C_2 C_3 \left(\tfrac{\ell+1}{2}\right)^{-(\frac{n-1}{2}+\frac{\ell+\mu}{2(\ell+1)})p}$. By using the last  estimate for $I_{j+1}(t,\varrho)$, we are now ready to estimate the integral $J_{j+1}(t)$. For $t\geqslant A_\ell^{-1}((8a_j+1)R+4(2\alpha_j+1)\phi_\ell(1))$ we get
\begin{align*}
 J_{j+1}(t) & \geqslant \int_0^{A_\ell(t)-(8a_j+1)R-4(2\alpha_j+1)\phi_\ell(1)} \varrho^{(n-1)[1-\frac{p}{2}]_{+}} \frac{\left((\phi_\ell(t)+R_1)^2-\varrho^2\right)^{-\gamma p}}{(A_\ell(t)+R-\varrho)^{\frac{n-1}{2}p}} \left(I_{j+1}(t,\varrho)\right)^p \mathrm{d}\varrho \\
 & \geqslant C_4^p\int_0^{A_\ell(t)-(8a_j+1)R-4(2\alpha_j+1)\phi_\ell(1)} \varrho^{(n-1)[1-\frac{p}{2}]_{+}} \frac{\left((\phi_\ell(t)+R_1)^2-\varrho^2\right)^{-\gamma p}}{(A_\ell(t)+R-\varrho)^{\frac{n-1}{2}p}} \\
 &  \qquad \qquad \times \big(A_\ell(t)-\varrho-R+2\phi_\ell(1)\big)^{-\left(\frac{n-1}{2}+\frac{\ell+\mu}{2(\ell+1)}\right) p^2+\left(n-1+\frac{r_2+2}{\ell+1} \right)p} 
  \\
 &  \qquad \qquad \times \left(\ln\left(\frac{A_\ell(t)-\varrho+\left(\frac{4(\theta a_j-1)}{1-\theta}-1\right)R+\frac{4(\theta \alpha_j+2)}{1-\theta}\phi_\ell(1)}{\frac{4(a_j-1)}{1-\theta}R+\frac{ 4(\alpha_j+2)}{1-\theta}\phi_\ell(1)}\right)\right)^{\frac{p^{j+2}-p}{p-1}}  \mathrm{d}\varrho.
\end{align*}
Next, we observe that the factor $ \left((\phi_\ell(t)+R_1)^2-\varrho^2\right)^{-\gamma p}(A_\ell(t)+R-\varrho)^{-\frac{n-1}{2}p}$ can be estimated from below exactly as we did in Subsection \ref{Subsection base case} (as $j$ grows, the domain of integration in $J_{j+1}(t)$ shrinks). 

Therefore, for $t\geqslant A_\ell^{-1}((8a_j+1)R+4(2\alpha_j+1)\phi_\ell(1))\geqslant A_\ell^{-1}(a_0R+\alpha_0\phi_\ell(1))$ the following estimate holds
\begin{align*}
\frac{\left((\phi_\ell(t)+R_1)^2-\varrho^2\right)^{-\gamma p}}{(A_\ell(t)+R-\varrho)^{\frac{n-1}{2}p}}\geqslant \widetilde{B} t^{-\frac{\ell p}{2}-\frac{1-\sqrt{\delta}}{2}p}\big(A_\ell(t)-\varrho-R+2\phi_\ell(1)\big)^{-\left(\frac{n-1}{2}+\gamma\right)p},
\end{align*} where $\widetilde{B}$ is defined in \eqref{def widetildeB}. By \eqref{equation for p crit}, we get that 
\begin{align*}
-\left(\tfrac{n-1}{2}+\tfrac{\ell+\mu}{2(\ell+1)}\right) p^2+\left(n-1+\tfrac{r_2+2}{\ell+1} \right)p-\left(\tfrac{n-1}{2}+\gamma\right)p  & =  -\left(\tfrac{n-1}{2}+\tfrac{\ell+\mu}{2(\ell+1)}\right) p^2+\left(\tfrac{n-1}{2}+\tfrac{r_2+2}{\ell+1}-\gamma \right)p \\
 & =  -\left(\tfrac{n-1}{2}+\tfrac{\ell+\mu}{2(\ell+1)}\right) p^2+\left(\tfrac{n+1}{2}+\tfrac{\mu-3\ell}{2(\ell+1)}\right)p  =-1.
\end{align*} Consequently,  for $t\geqslant A_\ell^{-1}((8a_j+1)R+4(2\alpha_j+1)\phi_\ell(1))$ we find that
\begin{align*}
J_{j+1}(t)\geqslant C_4^p \widetilde{B} \, t^{-\frac{\ell p}{2}-\frac{1-\sqrt{\delta}}{2}p} 
& \int_0^{A_\ell(t)-(8a_j+1)R-4(2\alpha_j+1)\phi_\ell(1)} \varrho^{(n-1)[1-\frac{p}{2}]_{+}}  \big(A_\ell(t)-\varrho-R+2\phi_\ell(1)\big)^{-1} 
  \\
 & \times \left(\ln\left(\frac{A_\ell(t)-\varrho+\left(\frac{4(\theta a_j-1)}{1-\theta}-1\right)R+\frac{4(\theta \alpha_j+2)}{1-\theta}\phi_\ell(1)}{\frac{4(a_j-1)}{1-\theta}R+\frac{ 4(\alpha_j+2)}{1-\theta}\phi_\ell(1)}\right)\right)^{\frac{p^{j+2}-p}{p-1}}  \mathrm{d}\varrho.
\end{align*} In the next step, we increase the the lower extreme of integration to $\theta(A_\ell(t)-(8a_j+1)R-4(2\alpha_j+1)\phi_\ell(1))$, obtaining for $t\geqslant A_\ell^{-1}((8a_j+1)R+4(2\alpha_j+1)\phi_\ell(1))$
\begin{align}
J_{j+1}(t)\geqslant C_4^p \widetilde{B} \theta^{(n-1)[1-\frac{p}{2}]_{+}} \, t^{-\frac{\ell p}{2}-\frac{1-\sqrt{\delta}}{2}p}  (A_\ell(t)-(8a_j+1)R-4(2\alpha_j+1)\phi_\ell(1))^{(n-1)[1-\frac{p}{2}]_{+}} \widetilde{J}_{j+1}(t), \label{final estimate Jj+1(t)}
\end{align} where 
\begin{align*}
\widetilde{J}_{j+1}(t) \doteq  &\int_{\theta(A_\ell(t)-(8a_j+1)R-4(2\alpha_j+1)\phi_\ell(1))}^{A_\ell(t)-(8a_j+1)R-4(2\alpha_j+1)\phi_\ell(1)} \big(A_\ell(t)-\varrho-R+2\phi_\ell(1)\big)^{-1} 
  \\
  & \qquad \times  \left(\ln\left(\frac{A_\ell(t)-\varrho+\left(\frac{4(\theta a_j-1)}{1-\theta}-1\right)R+\frac{4(\theta \alpha_j+2)}{1-\theta}\phi_\ell(1)}{\frac{4(a_j-1)}{1-\theta}R+\frac{ 4(\alpha_j+2)}{1-\theta}\phi_\ell(1)}\right)\right)^{\frac{p^{j+2}-p}{p-1}}  \mathrm{d}\varrho
\end{align*}

Our next goal is to estimate from below the integral $\widetilde{J}_{j+1}(t)$.
We begin by decreasing the argument of the logarithmic factor in the following way:
\begin{align}
\frac{A_\ell(t)-\varrho+\!\left(\frac{4(\theta a_j-1)}{1-\theta}-1\right) \! R+\frac{4(\theta \alpha_j+2)}{1-\theta}\phi_\ell(1)}{\frac{4(a_j-1)}{1-\theta}R+\frac{ 4(\alpha_j+2)}{1-\theta}\phi_\ell(1)} & \geqslant
\frac{A_\ell(t)-\varrho+\!\left(\frac{4(\theta a_j-1)}{1-\theta}-1-4a_j\right)\!R+\left(\frac{4(\theta \alpha_j+2)}{1-\theta}-4(\alpha_j+1)\right)\!\phi_\ell(1)}{\frac{4(a_j-1)}{1-\theta}R+\frac{ 4(\alpha_j+2)}{1-\theta}\phi_\ell(1)} \notag \\ & =
\frac{A_\ell(t)-\varrho+\frac{4(2\theta-1) a_j-5+\theta}{1-\theta}R+\frac{4(2\theta-1) \alpha_j+4(1+\theta)}{1-\theta}\phi_\ell(1)}{\frac{4(a_j-1)}{1-\theta}R+\frac{ 4(\alpha_j+2)}{1-\theta}\phi_\ell(1)} . \label{estimate log term}
\end{align} Thanks to the assumptions $\theta\in (\frac{1}{2},1)$ and \eqref{a0, alpha0 conditions}, it follows that $a_j$ and $\alpha_j$ satisfy for any $j\in\mathbb{N}$
\begin{align}\label{l.b. aj & alphaj}
a_j\geqslant \frac{1}{2\theta -1}, \qquad  \alpha_j\geqslant -\frac{1+3\theta}{2(2\theta-1)}.
\end{align} The condition on $\{\alpha_j\}_{j\in\mathbb{N}}$ is obviously true, since this is a sequence of positive real numbers, while the condition on $\{a_j\}_{j\in\mathbb{N}}$ is valid since this sequence is increasing and we assumed its validity for $j=0$ in \eqref{a0, alpha0 conditions}.
Therefore, by  \eqref{l.b. aj & alphaj} it follows that
\begin{align}\label{estimate term power -1}
A_\ell(t)-\varrho-R+2\phi_\ell(1)\leqslant A_\ell(t)-\varrho+\tfrac{4(2\theta-1) a_j-5+\theta}{1-\theta}R+\tfrac{4(2\theta-1) \alpha_j+4(1+\theta)}{1-\theta}\phi_\ell(1)
\end{align} Hence, for $t\geqslant A_\ell^{-1}((8a_j+1)R+4(2\alpha_j+1)\phi_\ell(1))$, combining \eqref{estimate log term} and \eqref{estimate term power -1}, we find that $\widetilde{J}_{j+1}(t)$  can be estimated from below by
\begin{align}
 & \tfrac{p-1}{p^{j+2}-1}\!\!\left[-\!\left(\!\ln\left(\!\frac{A_\ell(t)-\varrho+\frac{4(2\theta-1) a_j-5+\theta}{1-\theta}R+\frac{4(2\theta-1) \alpha_j+4(1+\theta)}{1-\theta}\phi_\ell(1)}{\frac{4(a_j-1)}{1-\theta}R+\frac{ 4(\alpha_j+2)}{1-\theta}\phi_\ell(1)}\right)\!\right)^{\frac{p^{j+2}-1}{p-1}} \right]^{\varrho=A_\ell(t)-(8a_j+1)R-4(2\alpha_j+1)\phi_\ell(1)}_{\varrho=\theta(A_\ell(t)-(8a_j+1)R-4(2\alpha_j+1)\phi_\ell(1))} \notag\\
 & = \tfrac{p-1}{p^{j+2}-1}\left(\ln\left(\frac{A_\ell(t)+\left(\frac{4(2\theta-1) a_j-5+\theta}{(1-\theta)^2}+\frac{\theta(8a_j+1)}{1-\theta}\right)R+\left(\frac{4(2\theta-1) \alpha_j+4(1+\theta)}{(1-\theta)^2}+\frac{4\theta(2\alpha_j+1)}{1-\theta}\right)\phi_\ell(1)}{\frac{4(a_j-1)}{(1-\theta)^2}R+\frac{ 4(\alpha_j+2)}{(1-\theta)^2}\phi_\ell(1)}\right)\right)^{\frac{p^{j+2}-1}{p-1}}.\label{final estimate J tilde j+1(t)}
\end{align}
Finally, we combine \eqref{l.b.e. |u(t)|^p_L^p begin inductive step}, \eqref{final estimate Jj+1(t)} and \eqref{final estimate J tilde j+1(t)}, obtaining for $t\geqslant A_\ell^{-1}((8a_j+1)R+4(2\alpha_j+1)\phi_\ell(1))$
\begin{align}
\| u(t,\cdot)\|^p_{L^p(\mathbb{R}^n)} &  \geqslant  KC_4^p \widetilde{B} \theta^{(n-1)[1-\frac{p}{2}]_{+}} \tfrac{p-1}{p^{j+2}-1} B_j^p \varepsilon^{p^{j+3}} t^{-\frac{\ell+\mu }{2}p}    \frac{(A_\ell(t)-(8a_j+1)R-4(2\alpha_j+1)\phi_\ell(1))^{(n-1)[1-\frac{p}{2}]_{+}} }{(A_\ell(t)+R)^{(n-1)[1-\frac{p}{2}]_{-}} } \notag \\ &\times \left(\ln\left(\frac{A_\ell(t)+\left(\frac{4(2\theta-1) a_j-5+\theta}{(1-\theta)^2}+\frac{\theta(8a_j+1)}{1-\theta}\right)R+\left(\frac{4(2\theta-1) \alpha_j+4(1+\theta)}{(1-\theta)^2}+\frac{4\theta(2\alpha_j+1)}{1-\theta}\right)\phi_\ell(1)}{\frac{4(a_j-1)}{(1-\theta)^2}R+\frac{ 4(\alpha_j+2)}{(1-\theta)^2}\phi_\ell(1)}\right)\right)^{\frac{p^{j+2}-1}{p-1}}. \label{l.b.e. |u(t)|^p L^p log almost j+1} 
\end{align} Since we want to obtain \eqref{(j+1)-th low bound ||u||^p with log} for $j+1$ from \eqref{l.b.e. |u(t)|^p L^p log almost j+1}, we set $D\doteq KC_4^p \widetilde{B} \theta^{(n-1)[1-\frac{p}{2}]_{+}}(1-\frac{1}{p})$ and impose the recursive relations  $B_{j+1}= Dp^{-(j+1)}B_j^p$, which is exactly \eqref{recursive relation Bj}, and 
\begin{align}
\frac{a_{j+1}-1}{1-\theta}&= \frac{4(a_j-1)}{(1-\theta)^2} \quad \Rightarrow \quad  a_{j+1}= 1+\frac{4(a_j-1)}{1-\theta}, \label{rec rel aj} \\
\frac{\alpha_{j+1}+2}{1-\theta}&= \frac{4(\alpha_j+2)}{(1-\theta)^2} \quad \Rightarrow \quad  \alpha_{j+1}= -2+\frac{4(\alpha_j+2)}{1-\theta}.  \label{rec rel alphaj} 
\end{align}
Clearly, \eqref{rec rel aj} and \eqref{rec rel alphaj}  tell us that $\{a_j-1\}_{j\in\mathbb{N}}$ and $\{\alpha_j+2\}_{j\in\mathbb{N}}$ are geometric sequences, hence the representations in \eqref{def aj} and \eqref{def alphaj}. Summarizing the previous notations, we may rewrite \eqref{l.b.e. |u(t)|^p L^p log almost j+1} as 
\begin{align}
\| u(t,\cdot)\|^p_{L^p(\mathbb{R}^n)} & \geqslant  B_{j+1} \varepsilon^{p^{j+3}} t^{-\frac{\ell+\mu }{2}p}    \frac{(A_\ell(t)-(8a_j+1)R-4(2\alpha_j+1)\phi_\ell(1))^{(n-1)[1-\frac{p}{2}]_{+}} }{(A_\ell(t)+R)^{(n-1)[1-\frac{p}{2}]_{-}} } \notag \\ & \quad \times \left(\ln\left(\frac{A_\ell(t)+\left(\frac{4(2\theta-1) a_j-5+\theta}{(1-\theta)^2}+\frac{\theta(8a_j+1)}{1-\theta}\right)R+\left(\frac{4(2\theta-1) \alpha_j+4(1+\theta)}{(1-\theta)^2}+\frac{4\theta(2\alpha_j+1)}{1-\theta}\right)\phi_\ell(1)}{\frac{a_{j+1}-1}{1-\theta}R+\frac{ \alpha_{j+1}+2}{1-\theta}\phi_\ell(1)}\right)\right)^{\frac{p^{j+2}-1}{p-1}} \label{l.b.e. |u(t)|^p L^p log really almost j+1} 
\end{align} for $t\geqslant A_\ell^{-1}((8a_j+1)R+4(2\alpha_j+1)\phi_\ell(1))$. Thus, in order to obtain \eqref{(j+1)-th low bound ||u||^p with log} for $j+1$ from \eqref{l.b.e. |u(t)|^p L^p log really almost j+1} we just have to check that base of the $(n-1)[1-\frac{p}{2}]_{+}$ power may be decreased as follows:
\begin{align}\label{final est (n-1)[1-p/2]+ power}
A_\ell(t)-(8a_j+1)R-4(2\alpha_j+1)\phi_\ell(1) \geqslant A_\ell(t)-a_{j+1}R-\alpha_{j+1}\phi_\ell(1),
\end{align}
and that the numerator of the argument of the logarithmic factor may be decreased too, more precisely, the following inequality must be fulfilled:
\begin{align} \label{final est log term}
A_\ell(t)+\left(\tfrac{4(2\theta-1) a_j-5+\theta}{(1-\theta)^2}+\tfrac{\theta(8a_j+1)}{1-\theta}\right)R+\left(\tfrac{4(2\theta-1) \alpha_j+4(1+\theta)}{(1-\theta)^2}+\tfrac{4\theta(2\alpha_j+1)}
{1-\theta}\right)\phi_\ell(1)\geqslant  A_\ell(t)+\tfrac{\theta a_{j+1}-1}{1-\theta}R+\tfrac{\theta \alpha_{j+1}+2}{1-\theta}\phi_\ell(1).
\end{align}
Elementary computations show that the conditions in \eqref{l.b. aj & alphaj} guarantee the validity of both \eqref{final est (n-1)[1-p/2]+ power} and \eqref{final est log term}. Therefore, for $t\geqslant A_\ell^{-1}(a_{j+1}R+\alpha_{j+1}\phi_\ell(1))\geqslant A_\ell^{-1}((8a_j+1)R+4(2\alpha_j+1)\phi_\ell(1))$ we proved that \eqref{(j+1)-th low bound ||u||^p with log} is satisfied for $j+1$. This completes the inductive proof of \eqref{(j+1)-th low bound ||u||^p with log} for any $j\in\mathbb{N}$.
\end{proof}

\subsection{Improved lower bound estimates for the spatial average of the solution} \label{Subsection log lb U}

In this subsection, we derive a sequence of lower bound estimates for the functional $U(t)$ that improve \eqref{1st low bound ||u||^p} thanks to the presence of logarithmic factors. Since we assumed $u_0,u_1$ to be nonnegative, we have that $U^{\mathrm{lin}}(t)\geqslant 0$ for any $t\geqslant 1$. Thus, from \eqref{Representation for U} we get 
\begin{align*}
U(t) & \geqslant  t^{-r_1} \int_1^t s^{r_1-r_2-1} \int_1^s \tau^{r_2+1} \, \|u(\tau,\cdot)\|^p_{L^p(\mathbb{R}^n)} \, \mathrm{d}\tau \, \mathrm{d}s
\end{align*}
for any $t\geqslant 1$. If we plug the lower bound estimate \eqref{(j+1)-th low bound ||u||^p with log} into the previous inequality, we arrive at 
\begin{align*}
U(t)  \geqslant  B_j\varepsilon^{p^{j+2}} t^{-r_1} \int_{A_\ell^{-1}(a_j R+\alpha_j \phi_\ell(1))}^t s^{r_1-r_2-1} & \int_{A_\ell^{-1}(a_j R+\alpha_j \phi_\ell(1))}^s \tau^{r_2+1-\frac{\ell+\mu }{2}p}\frac{(A_\ell(\tau)-a_jR-\alpha_j\phi_\ell(1))^{(n-1)[1-\frac{p}{2}]_{+}}}{(A_\ell(\tau)+R)^{(n-1)[1-\frac{p}{2}]_{-}}}  \\ &  \times  \left(\ln\left(\frac{A_\ell(\tau)+\frac{\theta a_j-1}{1-\theta}R+\frac{\theta \alpha_j+2}{1-\theta}\phi_\ell(1)}{\frac{a_j-1}{1-\theta}R+\frac{ \alpha_j+2}{1-\theta}\phi_\ell(1)}\right)\right)^{\frac{p^{j+1}-1}{p-1}} \, \mathrm{d}\tau \, \mathrm{d}s.
\end{align*}
Recalling that
\begin{align*}
A_\ell(\tau)-a_jR-\alpha_j\phi_\ell(1)\geqslant \tfrac 12 \phi_\ell(\tau) \quad & \Leftrightarrow \quad \tau \geqslant A_\ell^{-1}(2a_j R+ (2\alpha_j+1)\phi_\ell(1)) \\
A_\ell(\tau)+R \leqslant 2 \phi_\ell(\tau) \quad & \Leftrightarrow \quad \tau \geqslant A_\ell^{-1} (R-2\phi_\ell(1))
\end{align*}  we obtain for $t\geqslant A_\ell^{-1}(2a_j R+ (2\alpha_j+1)\phi_\ell(1)) $
\begin{align*}
U(t)  & \geqslant 2^{-(n-1)|1-\frac p2|}(\ell+1)^{(n-1)(\frac p2 -1)} B_j\varepsilon^{p^{j+2}} t^{-r_1} \int_{A_\ell^{-1}(2a_j R+(2\alpha_j+1) \phi_\ell(1))}^t s^{r_1-r_2-1}  \\ & \quad \times \int_{A_\ell^{-1}(2a_j R+(2\alpha_j+1) \phi_\ell(1))}^s \tau^{r_2+1-\frac{\ell+\mu }{2}p+(n-1)(\ell+1)(1-\frac{p}{2})} \left(\ln\left(\frac{A_\ell(\tau)+\frac{\theta a_j-1}{1-\theta}R+\frac{\theta \alpha_j+2}{1-\theta}\phi_\ell(1)}{\frac{a_j-1}{1-\theta}R+\frac{ \alpha_j+2}{1-\theta}\phi_\ell(1)}\right)\right)^{\frac{p^{j+1}-1}{p-1}} \, \mathrm{d}\tau \, \mathrm{d}s.
\end{align*} For $s\geqslant A_\ell^{-1}(4a_jR+2(2\alpha_j+1)\phi_\ell(1))$, we can increase the lower extreme of integration in the $\tau$-integral from $A_\ell^{-1}(2a_j R+(2\alpha_j+1) \phi_\ell(1))$ to $A_\ell^{-1}(\frac{1}{2}A_\ell(s))$. Hence, for $t\geqslant A_\ell^{-1}(4a_jR+2(2\alpha_j+1)\phi_\ell(1))$ we have
\begin{align}
U(t)  & \geqslant 2^{-(n-1)|1-\frac p2|}(\ell+1)^{(n-1)(\frac p2 -1)} B_j\varepsilon^{p^{j+2}} t^{-r_1} \int_{A_\ell^{-1}(4a_jR+2(2\alpha_j+1)\phi_\ell(1))}^t s^{r_1-r_2-1-\left((n-1)(\ell+1)+\ell+\mu)\right)\frac{p}{2}}  \notag \\ &\qquad \times \int_{A_\ell^{-1}\left(\frac{1}{2}A_\ell(s)\right)}^s \tau^{r_2+1+(n-1)(\ell+1)}   \left(\ln\left(\frac{A_\ell(\tau)+\frac{\theta a_j-1}{1-\theta}R+\frac{\theta \alpha_j+2}{1-\theta}\phi_\ell(1)}{\frac{a_j-1}{1-\theta}R+\frac{ \alpha_j+2}{1-\theta}\phi_\ell(1)}\right)\right)^{\frac{p^{j+1}-1}{p-1}} \, \mathrm{d}\tau \, \mathrm{d}s \notag \\
 & \geqslant 2^{-(n-1)|1-\frac p2|}(\ell+1)^{(n-1)(\frac p2 -1)} B_j\varepsilon^{p^{j+2}} t^{-r_1} \int_{A_\ell^{-1}(4a_jR+2(2\alpha_j+1)\phi_\ell(1))}^t s^{r_1-r_2-1-\left((n-1)(\ell+1)+\ell+\mu)\right)\frac{p}{2}}  \notag \\ &\qquad \times  \left(\ln\left(\frac{A_\ell(s)+\frac{2(\theta a_j-1)}{1-\theta}R+\frac{2(\theta \alpha_j+2)}{1-\theta}\phi_\ell(1)}{\frac{2(a_j-1)}{1-\theta}R+\frac{ 2(\alpha_j+2)}{1-\theta}\phi_\ell(1)}\right)\right)^{\frac{p^{j+1}-1}{p-1}}  \int_{A_\ell^{-1}\left(\frac{1}{2}A_\ell(s)\right)}^s \tau^{r_2+1+(n-1)(\ell+1)}  \, \mathrm{d}\tau \, \mathrm{d}s. \label{lb U(t) int with ln s term}
\end{align} Let us estimate from  below the last $\tau$-integral, for  $s\geqslant A_\ell^{-1}(4a_jR+2(2\alpha_j+1)\phi_\ell(1))$ it holds
\begin{align}
\int_{A_\ell^{-1}\left(\frac{1}{2}A_\ell(s)\right)}^s \tau^{r_2+1+(n-1)(\ell+1)}  \, \mathrm{d}\tau  &= \frac{s^{r_2+2+(n-1)(\ell+1)}-\left(A_\ell^{-1}\left(\tfrac{1}{2} A_\ell(s)\right)\right)^{r_2+2+(n-1)(\ell+1)}}{r_2+2+(n-1)(\ell+1)} \notag \\ & = (r_2+2+(n-1)(\ell+1))^{-1}\left(1-\left(\frac{1+s^{-(\ell+1)}}{2}\right)^{\frac{r_2+2}{\ell+1}+n-1}\right) \, s^{r_2+2+(n-1)(\ell+1)} \notag \\
& \geqslant (r_2+2+(n-1)(\ell+1))^{-1} \beta_j \, s^{r_2+2+(n-1)(\ell+1)}, \label{tau integral U(t) lb}
\end{align} where we used that 
\begin{align*}
A_\ell^{-1}\left(\tfrac{1}{2} A_\ell(s)\right)= \left(\frac{s^{\ell+1}+1}{2}\right)^{\frac{1}{\ell+1}}
\end{align*} and we introduced the increasing sequence $\{\beta_j\}_{j\in\mathbb{N}}$ whose $j$-th term is given by
\begin{align*}
\beta_j &\doteq  1-\left(\frac{1+(A_\ell^{-1}(4a_jR+2(2\alpha_j+1)\phi_\ell(1)))^{-(\ell+1)}}{2}\right)^{\frac{r_2+2}{\ell+1}+n-1} \\
 &= 1-\left(\frac{1}{2}\left(1+\frac{\phi_\ell(1)}{4a_jR+(4\alpha_j+3)\phi_\ell(1)}\right)\right)^{\frac{r_2+2}{\ell+1}+n-1}.
\end{align*} Combining \eqref{lb U(t) int with ln s term} and \eqref{tau integral U(t) lb}, for $t\geqslant A_\ell^{-1}(4a_jR+2(2\alpha_j+1)\phi_\ell(1))$ we obtain
\begin{align*}
U(t)    \geqslant \widehat{M}  \beta_j B_j\varepsilon^{p^{j+2}} t^{-r_1} &\int_{A_\ell^{-1}(4a_jR+2(2\alpha_j+1)\phi_\ell(1))}^t s^{r_1+1-\frac{\ell+\mu}{2}p+(n-1)(\ell+1)(1-\frac{p}{2})}  \\ & \qquad \times  \left(\ln\left(\frac{A_\ell(s)+\frac{2(\theta a_j-1)}{1-\theta}R+\frac{2(\theta \alpha_j+2)}{1-\theta}\phi_\ell(1)}{\frac{2(a_j-1)}{1-\theta}R+\frac{ 2(\alpha_j+2)}{1-\theta}\phi_\ell(1)}\right)\right)^{\frac{p^{j+1}-1}{p-1}} \, \mathrm{d}s,
\end{align*} where $\widehat{M}\doteq 2^{-(n-1)|1-\frac p2|}(\ell+1)^{(n-1)(\frac p2 -1)}(r_2+2+(n-1)(\ell+1))^{-1}$.

Analogously to what we did for the $\tau$-integral, taking $t\geqslant A_\ell^{-1}(8a_jR+4(2\alpha_j+1))$ we can increase the lower bound of the domain of integration in the previous $s$-integral from $A_\ell^{-1}(4a_jR+2(2\alpha_j+1)\phi_\ell(1))$ to $A_\ell^{-1}(\frac{1}{2}A_\ell(t))$, obtaining
\begin{align*}
U(t)    &\geqslant \widehat{M} \beta_j B_j\varepsilon^{p^{j+2}} t^{-r_1-\frac{\ell+\mu}{2}p-(n-1)(\ell+1)\frac{p}{2}}  \\ &\qquad \times \int_{A_\ell^{-1}(\frac{1}{2}A_\ell(t))}^t s^{r_1+1+(n-1)(\ell+1)}  \left(\ln\left(\frac{A_\ell(s)+\frac{2(\theta a_j-1)}{1-\theta}R+\frac{2(\theta \alpha_j+2)}{1-\theta}\phi_\ell(1)}{\frac{2(a_j-1)}{1-\theta}R+\frac{ 2(\alpha_j+2)}{1-\theta}\phi_\ell(1)}\right)\right)^{\frac{p^{j+1}-1}{p-1}} \, \mathrm{d}s \\
&\geqslant \widehat{M}   \beta_j B_j\varepsilon^{p^{j+2}} t^{-r_1-\frac{\ell+\mu}{2}p-(n-1)(\ell+1)\frac{p}{2}}  \\ &\qquad \times \left(\ln\left(\frac{A_\ell(t)+\frac{4(\theta a_j-1)}{1-\theta}R+\frac{4(\theta \alpha_j+2)}{1-\theta}\phi_\ell(1)}{\frac{4(a_j-1)}{1-\theta}R+\frac{ 4(\alpha_j+2)}{1-\theta}\phi_\ell(1)}\right)\right)^{\frac{p^{j+1}-1}{p-1}} \int_{A_\ell^{-1}(\frac{1}{2}A_\ell(t))}^t s^{r_1+1+(n-1)(\ell+1)}   \, \mathrm{d}s \\
&\geqslant  \widetilde{M}  \beta_j \widetilde{\beta}_j B_j\varepsilon^{p^{j+2}} t^{2-\frac{\ell+\mu}{2}p+(n-1)(\ell+1)(1-\frac{p}{2})}   \left(\ln\left(\frac{A_\ell(t)+\frac{4(\theta a_j-1)}{1-\theta}R+\frac{4(\theta \alpha_j+2)}{1-\theta}\phi_\ell(1)}{\frac{4(a_j-1)}{1-\theta}R+\frac{ 4(\alpha_j+2)}{1-\theta}\phi_\ell(1)}\right)\right)^{\frac{p^{j+1}-1}{p-1}} 
\\
&\geqslant  \widetilde{M}  \beta_j \widetilde{\beta}_j B_j\varepsilon^{p^{j+2}} t^{2-\frac{\ell+\mu}{2}p+(n-1)(\ell+1)(1-\frac{p}{2})}   \left(\ln\left(\frac{A_\ell(t)}{\frac{4(a_j-1)}{1-\theta}R+\frac{ 4(\alpha_j+2)}{1-\theta}\phi_\ell(1)}\right)\right)^{\frac{p^{j+1}-1}{p-1}} ,
\end{align*} where the multiplicative constant is $\widetilde{M} \doteq \widehat{M} \left(r_1+2+(n-1)(\ell+1)\right)^{-1} $, the $j$-th term of the increasing sequence $\{\widetilde{\beta}_j \}_{j\in\mathbb{N}}$ is defined by
\begin{align*}
\widetilde{\beta}_j \doteq  1-\left(\frac{1}{2}\left(1+\frac{\phi_\ell(1)}{8a_jR+(8\alpha_j+5)\phi_\ell(1)}\right)\right)^{\frac{r_1+2}{\ell+1}+n-1}
\end{align*} and in the last inequality we used that $\theta>\frac{1}{2}$ and $a_j\geqslant 2$, $\alpha_j\geqslant 0$ for any $j\in\mathbb{N}$.

By elementary computations, we find that
\begin{align*}
&\ln\left(\frac{A_\ell(t)}{\frac{4(a_j-1)}{1-\theta}R+\frac{ 4(\alpha_j+2)}{1-\theta}\phi_\ell(1)}\right) \geqslant\tfrac{1}{2}\ln A_\ell(t) &&  \Leftrightarrow \quad t\geqslant A_\ell^{-1}\left(\tfrac{16}{(1-\theta)^2}\big((a_j-1)R+(\alpha_j+2)\phi_\ell(1)\big)^2\right),  \\
 &\ln A_\ell(t) \geqslant \ln\left(\tfrac{t^{\ell+1}}{2(\ell+1)}\right) &&  \Leftrightarrow \quad  t\geqslant 2^{\frac{1}{\ell+1}},\\
& \ln\left(\tfrac{t^{\ell+1}}{2(\ell+1)} \right) \geqslant \tfrac{\ell+1}{2} \ln t && \Leftrightarrow \quad  t\geqslant  (2(\ell+1))^{\frac{1}{\ell+1}}.
\end{align*} We introduce the sequence $\{\sigma_j\}_{j\in\mathbb{N}}$ such that
\begin{align} \label{def sigmaj} 
\sigma_j\doteq \max \left\{A_\ell^{-1}(8a_jR+4(2\alpha_j+1)),A_\ell^{-1}\left(\tfrac{16}{(1-\theta)^2}\big((a_j-1)R+(\alpha_j+2)\phi_\ell(1)\big)^2\right),2^{\frac{1}{\ell+1}},(2(\ell+1))^{\frac{1}{\ell+1}} \right\}.
\end{align}
Summarizing, for any $j\in\mathbb{N}$ and for any $t\geqslant \sigma_j$ we proved that 
\begin{align} \label{improved lb seq U(t) log }
U(t) \geqslant M \beta_j \widetilde{\beta}_j Q^{p^{j+1}} B_j \varepsilon^{p^{j+2}} (\ln t)^\frac{p^{j+1}-1}{p-1} t^{2-\frac{\ell+\mu}{2}p+(n-1)(\ell+1)(1-\frac{p}{2})}, 
\end{align} where $M\doteq \widetilde{M} (\frac{\ell+1}{4})^{-\frac{1}{p-1}}$ and $Q\doteq (\frac{\ell+1}{4})^{\frac{1}{p-1}} $. 

\subsection{Proof of Theorem \ref{Thm main}}

We complete the proof of the main theorem, by applying Lemma \ref{Lemma Kato-type} to the function $U(t)$.
By \eqref{ODE for U}, applying H\"older's inequality, it follows that $U(t)$ satisfies the ODI
\begin{align}\label{ODI U(t)}
U''(t)+\mu t^{-1} U'(t)+\nu^2 t^{-2} U(t) \geqslant C t^{-n(\ell+1)(p-1)}(U(t))^p
\end{align} for any $t\in [1,T(\varepsilon))$, where $T(\varepsilon)$ denotes the lifespan of the local in time  solution $u$. In \cite[Subsection 2.2]{Pal21} it is proved the lower bound estimate
\begin{align*}
U(t)\geqslant K t^{2-\frac{\ell+\mu}{2}p+(n-1)(\ell+1)(1-\frac{p}{2})}
\end{align*} for any $t\in [T_2,T(\varepsilon))$, where $T_2\geqslant 2$, and for a positive suitable constant $K$. Actually, in Subsection \ref{Subsection log lb U} we improved the previous estimate by proving the sequence of lower bound estimates in \eqref{improved lb seq U(t) log }. Moreover, being $u_0,u_1$ nonnegative data and thanks to \eqref{integral assumption Cauchy data}, we have that $U(1),U'(1)\geqslant 0$ and $U'(1)+r_1U(1)>0$. Therefore, $U(t)$ satisfies the assumptions of Lemma \ref{Lemma Kato-type} for $a=2-\tfrac{\ell+\mu}{2}p+(n-1)(\ell+1)(1-\frac{p}{2})$ and $q=n(\ell+1)(p-1)$ since $p=p_{\mathrm{Str}}\left(n+\frac{\mu}{\ell+1},\ell\right)$ satisfies \eqref{equation for p crit}. Hereafter, $K_0$ and $\widetilde{T}_0$ denotes the constants associated to the function $U(t)$ according to the statement of Lemma \ref{Lemma Kato-type}.

Let us introduce the notation 
\begin{align*}
K_j(t,\varepsilon) \doteq M \beta_j \widetilde{\beta}_j Q^{p^{j+1}} B_j \varepsilon^{p^{j+2}} (\ln t)^\frac{p^{j+1}-1}{p-1}
\end{align*} for the function that multiplies the power term in \eqref{improved lb seq U(t) log }. Then, we can rewrite \eqref{improved lb seq U(t) log } as
\begin{align*}
U(t) \geqslant  K_j(t,\varepsilon)  \, t^{2-\frac{\ell+\mu}{2}p+(n-1)(\ell+1)(1-\frac{p}{2})} 
\end{align*} for $j\in\mathbb{N}$ and $t\geqslant \sigma_j$. Since $B_j$ fulfills the recursive relation \eqref{recursive relation Bj}, we have that
\begin{align*}
\ln B_j & = p \ln B_{j-1} -j \ln p +\ln D 
\\ &= p^2 \ln B_{j-2} -(j +(j-1)p)\ln p +(1+p)\ln D  \\ & = \ldots = p^j B_0-\ln p\sum_{k=0}^{j-1}(j-k)p^k+\ln D \sum_{k=0}^{j-1}p^k \\ & = p^j B_0-\frac{\ln p}{p-1}\left(\frac{p^{j+1}-1}{p-1}-(j+1)\right)+\ln D \, \frac{p^j-1}{p-1} \\
&=p^j\left(\ln B_0-\frac{p\ln p}{(p-1)^2}+\frac{\ln D}{p-1}\right)+(j+1)\frac{\ln p }{p-1}+\frac{\ln p}{(p-1)^2}-\frac{\ln D}{p-1} \\ &=p^j \ln E_0 +(j+1)\frac{\ln p }{p-1}+\frac{\ln p}{(p-1)^2}-\frac{\ln D}{p-1},
\end{align*} where $E_0\doteq B_0 p^{-p/(p-1)^2}D^{1/(p-1)}$. Therefore, we have
\begin{align*}
K_j(t,\varepsilon) & = \exp\left\{ \ln M+ \ln \left(\beta_j \widetilde{\beta}_j\right)+p^{j+1} \ln Q +\ln B_j+p^{j+1}\ln\varepsilon^{p}+\ln\left[ (\ln t)^\frac{p^{j+1}-1}{p-1}\right]\right\} \\
& = \exp\left\{p^{j+1}\left(\ln \varepsilon^p+ \ln Q+\tfrac{1}{p}E_0+\ln \left[(\ln t)^\frac{1}{p-1}\right]\right)+ \ln N+ \ln \left(\beta_j \widetilde{\beta}_j\right) +(j+1)\tfrac{\ln p }{p-1}-\tfrac{1}{p-1}\ln\ln t\right\} \\
& = \exp\left\{p^{j+1}\ln \left(E_1 \varepsilon^p (\ln t)^\frac{1}{p-1}\right)+ \ln N+ \ln \left(\beta_j \widetilde{\beta}_j\right) +(j+1)\tfrac{\ln p }{p-1}-\tfrac{1}{p-1}\ln\ln t\right\},
\end{align*} where $N\doteq M p^{1/(p-1)^2}D^{-1/(p-1)}$ and $E_1\doteq Q E_0^{1/p}$. Now, we introduce $L(t,\varepsilon)\doteq \ln \left(E_1 \varepsilon^p (\ln t)^\frac{1}{p-1}\right)$, so, for any $j\in\mathbb{N}$
\begin{align*}
K_j(t,\varepsilon) & = \exp\left\{p^{j+1}L(t,\varepsilon)+ \ln N+ \ln \left(\beta_j \widetilde{\beta}_j\right) +(j+1)\tfrac{\ln p }{p-1}-\tfrac{1}{p-1}\ln\ln t\right\}.
\end{align*}
Clearly, $L(t,\varepsilon)\geqslant 1$ if and only if $t\geqslant \exp\left(E_2 \varepsilon^{-p(p-1)}\right) \doteq T_0(\varepsilon)$ where $E_2\doteq (\mathrm{e}E_1^{-1})^{p-1}$. Consequently, for $t\geqslant T_0(\varepsilon)$ and for any $j\in\mathbb{N}$ it holds
\begin{align*}
K_j(t,\varepsilon) & \geqslant \exp\left\{p^{j+1}+ \ln N+ \ln \left(\beta_j \widetilde{\beta}_j\right) +(j+1)\tfrac{\ln p }{p-1}-\tfrac{1}{p-1}\ln\ln t\right\}.
\end{align*}
We consider the family of intervals $\{\mathcal{I}(j)\}_{j\in\mathbb{N}}$ with $\mathcal{I}(j)\doteq [\sigma_j,\sigma_{j+1}]$. 

Since $a_j,\alpha_j$ grow exponentially, cf. \eqref{def aj} and \eqref{def alphaj}, according to \eqref{def sigmaj} there exists $j_0=j_0(\ell,\theta,R,a_0,\alpha_0)\in\mathbb{N}$ such that $\forall j\geqslant j_0$: $\sigma_j=A_\ell^{-1}\left(\tfrac{16}{(1-\theta)^2}\big((a_j-1)R+(\alpha_j+2)\phi_\ell(1)\big)^2\right)$.
For any $j\in\mathbb{N}, j\geqslant j_0$ and for any $t\in\mathcal{I}(j)$ such that $t\geqslant T_0(\varepsilon)$ it holds
\begin{align*}
K_j(t,\varepsilon) & \geqslant \exp\left\{p^{j+1}+ \ln N+ \ln \left(\beta_j \widetilde{\beta}_j\right) +(j+1)\tfrac{\ln p }{p-1}-\tfrac{1}{p-1}\ln\ln \left(\tfrac{16}{(1-\theta)^2}\big((a_{j+1}-1)R+(\alpha_{j+1}+2)\phi_\ell(1)\big)^2\right)\right\}.
\end{align*} Since the last term in the exponent on the right-hand side, namely $\ln \ln \sigma_{j+1}$, grows logarithmically  with respect to $j$ and the term $\ln(\beta_j \widetilde{\beta}_j)$ is bounded, as $j\to \infty$ the right-hand side of the previous inequality tends to $+\infty$. Hence, we may find $J=J(n,\ell, \mu,\nu^2,R,p,a_0,\alpha_0,\theta)\in\mathbb{N}$, $J\geqslant j_0$ such that for any $j\in\mathbb{N}, j\geqslant J$ and for any $t\in\mathcal{I}(j)$ such that $t\geqslant T_0(\varepsilon)$: $K_j(t,\varepsilon)  \geqslant K_0$. Then, for  any  $t\geqslant \sigma_J$ such that $t\geqslant  T_0(\varepsilon)$ we have 
\begin{align}\label{l.b. U K0}
U(t)\geqslant K_0 ^{2-\frac{\ell+\mu}{2}p+(n-1)(\ell+1)(1-\frac{p}{2})}.
\end{align} Since $T_0(\varepsilon)$ is decreasing with respect to $\varepsilon$, there exists  $\varepsilon_0=\varepsilon_0(n,\ell,\mu,\nu^2,R,p,a_0,\alpha_0,\theta, J,u_0,u_1)>0$ such that $T_0(\varepsilon_0)\geqslant \max\{\sigma_J,\widetilde{T}_0\}$. Therefore, for any $\varepsilon\in(0,\varepsilon_0]$ the function $U(t)$ satisfies the ODI  \eqref{ODI U(t)} for any $t\in [1,T(\varepsilon))$ and the lower bound estimate \eqref{l.b. U K0} for any $t\in [T_0(\varepsilon),T(\varepsilon))$, so $U(t)$ must blow up in finite time and the following upper bound estimate for the lifespan holds
\begin{align*}
T(\varepsilon)\leqslant 2 \max\{T_0(\varepsilon),\sigma_J,\widetilde{T}_0\}=2 T_0(\varepsilon)=2 \exp\left(E_1\varepsilon^{-p(p-1)}\right) \leqslant \exp\left(E\varepsilon^{-p(p-1)}\right),
\end{align*} where $E\doteq 2E_1$ and assumed without loss of generality that $\varepsilon_0\leqslant (\frac{\ln 2}{E_1})^{-1/(p(p-1))}$. The last estimate is exactly \eqref{lifespan est thm main} and this completes the proof of Theorem \ref{Thm main}.
\begin{remark} We point out that the smallness condition $\varepsilon\in(0,\varepsilon_0]$ for the Cauchy data is necessary to get the upper bound estimate for the lifespan. Without this condition, provided that $u_0,u_1$ satisfy the  assumptions of Theorem \ref{Thm main}, we have that for $t\geqslant\max\{T_0(\varepsilon),\sigma_j\}$ the lower bound estimate for $U(t)$ is \eqref{l.b. U K0} and, then, $U(t)$ has to blow up in finite time according to Lemma \ref{Lemma Kato-type}.
\end{remark}

\section{Concluding remarks and open problems}

The result of the present paper together with those in \cite{Pal21} complete the picture of blow-up results for the semilinear Cauchy problem \eqref{semilinear EPDT} when $\mu,\nu^2$ satisfy $\delta\geqslant 0$. Concerning the global existence counterpart of our results, so far, only the one dimensional case has been completely investigated in \cite[Section 5]{DAbb21}. In particular, combining the global existence result in  \cite{DAbb21} with our results, we have that 
\begin{align}\label{critical exponent}
\max\left\{p_{\mathrm{Str}}\left(n+\frac{\mu}{\ell+1},\ell\right),p_{\mathrm{Fuj}}\left((\ell+1)n+\tfrac{\mu-1-\sqrt{\delta}}{2}\right)\right\}
\end{align} is the critical exponent when $n=1$. Here, by critical exponent, we mean the threshold value for the power $p$ of the semilinear term $|u|^p$ in \eqref{semilinear EPDT} that separates the blow-up region (below the critical exponent) from the global existence of small-data solutions region (above the critical exponent). For further details on critical exponents we refer to \cite{L18}. In space dimension $n\geqslant 2$ there are only partial results for the global existence of small-data solutions. More specifically, for the wave equation with scale-invariant damping and mass (that is, when $\ell =0$), we cite \cite{DAbb15,Wak14,DLR15,DabbLuc15,NPR16,Pal18odd,Pal18even,Lai20,LZ21}   while for the Tricomi equation (that is, when $\mu=\nu^2=0$), we refer to \cite{HWY17,HWY17d2,HWY16}. In all the results mentioned above, the critical exponent is a special case of \eqref{critical exponent}.

A further open problem concerns the sharpness of the upper bound of the lifespan estimates in the sub-critical case/ critical case. In particular, for the estimate \eqref{lifespan est thm main} it is well known to be optimal in the case of the critical semilinear wave equation (i.e. for $\ell=\mu=\nu^2=0$ and $p=p_{\mathrm{Str}}(n,0)$), cf. \cite{TakWak11,ZH14,WakYor19}. However, the optimality in the general case is an open question.

In our problem under consideration, we impose the initial data on $t_0=1$, which makes us to avoid the singularity caused by the coefficients of the damping and mass term, we may refer it as regular Cauchy problem. However, if the initial data are prescribed at $t_0=0$, which corresponds to the singular Cauchy problem, it seems the critical power is unclear till now. We refer to the more detailed discussion in \cite{DAbb21}.

Another challenging question concerns the blow-up of energy solutions or of weak solutions (in the sense of Definition \ref{def sol}) in the critical case $p=p_{\mathrm{Str}}(n+\frac{\mu}{\ell+1},\ell)$: indeed, in  Theorem \ref{Thm main} we worked just with classical solutions, as we used Yagdjian's integral representation formula.

\section*{Acknowledgments}

N.-A. Lai was partially supported by NSFC (12271487, 12171097). A. Palmieri is member of the \emph{Gruppo Nazionale per L’Analisi Matematica, la Probabilità e le loro Applicazioni} (GNAMPA) of the \emph{Instituto Nazionale di Alta Matematica} (INdAM). A. Palmieri has been partially supported by  INdAM - GNAMPA Project 2024 ``Modelli locali e non-locali con perturbazioni non-lineari'' CUP E53C23001670001 and by ERC Seeds UniBa Project ``NWEinNES'' CUP H93C23000730001. H. Takamura is partially supported by the Grant-in-Aid for Scientific Research (A) (No.22H00097), Japan Society for the
Promotion of Science.


\addcontentsline{toc}{chapter}{Bibliography}


\begin{thebibliography}{00}  




%



\bibitem{CGL21} Chiarello F.A., Girardi G., Lucente S.,
\newblock{Fujita modified exponent for scale invariant damped semilinear wave equations.}
\newblock{\em J. Evol. Equ.} {\bf 21}: 2735--2748 (2021).

\bibitem{Co56} Copson E.T.,
\newblock{On a regular Cauchy problem for the Euler-Poisson-Darboux equation.}
\newblock{\em  Proc. R. Soc. A}, {\bf 235}, 1203: 560--572 (1956).

\bibitem{DAbb15}    D'Abbicco M.,
\newblock {The threshold of effective damping for semilinear wave equations.}
\newblock {\em Math. Methods Appl. Sci.} {\bf 38}(6):1032--1045  (2015).

\bibitem{DAbb21}  D'Abbicco M.,
\newblock{Small data solutions for the Euler-Poisson-Darboux equation with a power nonlinearity.}
\newblock {\em J. Differential Equations} {\bf 286}(10):531--556 (2021).

\bibitem{DL13}  D'Abbicco M.,  Lucente S., 
\newblock{A modified test function method for damped wave equations.} \newblock {\em Adv. Nonlinear Stud.} {\bf 13}(4):867--892 (2013).


\bibitem{DabbLuc15}  D'Abbicco M.,  Lucente S., (2015). 
\newblock{NLWE with a special scale invariant damping in odd space dimension.}
\newblock{\em Discrete Contin. Dyn. Syst.}
\newblock{Dynamical systems, differential equations and applications.} 10th AIMS Conference. Suppl., 312--319, doi: 10.3934/proc.2015.0312.

\bibitem{DLR15}   D'Abbicco M.,  Lucente S., Reissig  M., 
\newblock {A shift in the Strauss exponent for semilinear wave equations with a not effective damping.}
\newblock {\em J. Differential Equations} {\bf 259}(10): 5040--5073 (2015).


\bibitem{DabbPal18} D'Abbicco M., Palmieri  A., 
\newblock{A note on $L^p-L^q$ estimates for semilinear critical dissipative Klein–Gordon equations.}
\newblock{\em J. Dyn. Differ. Equ.} {\bf 33}: 63--74 (2021). https://doi.org/10.1007/s10884-019-09818-2


\bibitem{Da} Darboux G.,
\newblock{\em Le\c{c}ons sur la th\'{e}orie g\'{e}n\'{e}rale des surfaces et les applications g\'{e}om\'{e}triques du calcul infinit\'{e}simal},
Gauthier-Villars et fils, Paris, 1896.


\bibitem{DL71} Delache S., Leray J.,
\newblock{Calcul de la solution \'{e}l\'{e}mentaire de l'op\'{e}rateur d’Euler-Poisson-Darboux et de l'op\'{e}rateur de Tricomi-Clairaut, hyperbolique, d'ordre 2.}
\newblock{\em Bulletin de la S. M. F.} {\bf 99}: 313--336 (1971). 

\bibitem{DW53}  Diaz J.B., Weinberger H.F., 
\newblock{A solution of the singular initial value problem for the Euler–Poisson–Darboux equation.}
\newblock{\em  Proc. Amer. Math. Soc.} {\bf 4}: 703--715 (1953). 

\bibitem{Eu} Euler L.,
\newblock{\em Institutiones calculi integralis}, Vol. III, Petropoli. 1770. Also in: Leonhardi Euleri Opera Omnia, Series
History of Mathematical Sciences, Birkh\"auser Basel, 1989.


\bibitem{Fuj66} Fujita H.,
\newblock{On the blowing up of solutions of the Cauchy problem for $u_t=\Delta u+ u^{1+\alpha}$.} {\em J. Fac. Sci. Univ. Tokyo} {\bf 13}: 109--124 (1966).


 
 
\bibitem{Geo97}  Georgiev V., Lindblad H., Sogge C.D.,
\newblock {Weighted Strichartz estimates and global existence for semi-linear wave equations,}
\newblock { Amer. J. Math.} {\bf 119}(6): 1291--1319 (1997).

\bibitem{GL22} Girardi G., Lucente S.,
\newblock{Lifespan estimates for a special quasilinear time-dependent damped wave equation}. \newblock{In Cerejeiras P., Reissig M., Sabadini I., Toft J. eds} \newblock{\em Current Trends in Analysis, its Applications and Computation.} \newblock{Trends in Mathematics} (2022).  https://doi.org/10.1007/978-3-030-87502-2$\_$61



\bibitem{Glas81B} Glassey R.T.,
\newblock {Finite-time blow-up for solutions of nonlinear wave equations,}
\newblock { Math Z.} {\bf 177}(3): 323--340 (1981).


\bibitem{Glas81}  Glassey R.T.,
\newblock {Existence in the large for $\square u = F(u)$ in two space dimensions,}
\newblock { Math Z.} {\bf 178}(2): 233--261 (1981).

 






\bibitem{HHP20} Hamouda M., Hamza M.A., Palmieri A.,
\newblock {A note on the nonexistence of global solutions to the semilinear wave equation with nonlinearity of derivative-type in the generalized Einstein--de Sitter spacetime.}  
\newblock{\em  Commun. Pure Appl. Anal.} {\bf 20}(11): 3703--3721 (2021). doi: 10.3934/cpaa.2021127



\bibitem{HWY17} He D.Y, Witt I., Yin H.C.,
\newblock{On the global solution problem for semilinear generalized Tricomi equations, I.}
\newblock{\em Calc. Var.} {\bf 56}, 21 (2017).


\bibitem{HWY16} He D.Y, Witt I., Yin H.C.,
\newblock{On the global solution problem of semilinear generalized Tricomi equations, II.}
\newblock{\em Pac. J. Math.} {\bf 314}(1): 29--80 (2021).

\bibitem{HWY17d2} He D.Y, Witt I., Yin H.C.,
\newblock{On semilinear Tricomi equations with critical exponents or in two space dimensions.}
\newblock{\em J. Differential Equations} {\bf 263}(12): 8102--8137 (2017).

\bibitem{HWY20} He D.Y, Witt I., Yin H.C.,
\newblock{On the Strauss index of semilinear Tricomi equation.}
 \newblock{\em Commun. Pure Appl. Anal.} {\bf 19}(10): 4817--4838 (2020). doi: 10.3934/cpaa.2020213 
 
 
\bibitem{Hel11} Helgason S.,
\newblock{\em Integral Geometry and Radon Transforms}
\newblock{Springer} , New York, 2011.



\bibitem{IS17}  Ikeda M., Sobajima M., 
\newblock{Life-span of solutions to semilinear wave equation with time-dependent critical damping for specially localized initial data.}
\newblock{\em Math. Ann.}  {\bf 372}(3/4):1017--1040 (2018).  



\bibitem{John79}  John F.,
\newblock {Blow-up of solutions of nonlinear wave equations in three space dimensions,}
\newblock{ Manuscripta Math.} {\bf 28}(1-3): 235--268 (1979).

	
\bibitem{Kato80} Kato T.,
\newblock{ Blow-up of solutions of some nonlinear hyperbolic equations,} \newblock{ Comm. Pure Appl. Math.} {\bf 33}(4) 501--505 (1980).
	



\bibitem{Lai20} Lai N.A.,
\newblock{Weighted $L^2-L^2$ estimate for wave equation in $\mathbf{R}^3$ and its applications,}
\newblock{The role of metrics in the theory of partial differential equations,} 
\newblock{\em Adv. Stud. Pure Math.} {\bf 85}, (2020): 269--279. DOI: 10.2969/aspm/08510269


\bibitem{LST20} Lai N.A., Schiavone N.M., Takamura  H.,
\newblock{Heat-like and wave-like lifespan estimates for solutions of semilinear damped wave equations via a Kato's type lemma.}
\newblock {\em J. Differential Equations} {\bf 269}(12): 11575--11620 (2020).

\bibitem{LTW17} Lai  N. A., Takamura  H., Wakasa,K., 
\newblock{Blow-up for semilinear wave equations with the scale invariant damping and super-Fujita exponent.}
\newblock{\em J. Differential Equations} {\bf 263}(9): 5377--5394 (2017).

\bibitem{LZ21} Lai N.A., Zhou Y.,
\newblock{Global existence for semilinear wave equations with scaling invariant damping in 3-D,}
\newblock{\em Nonlinear Anal.} {\bf 210}, 112392 (2021). https://doi.org/10.1016/j.na.2021.112392


\bibitem{LinSog96} Lindblad H., Sogge C.,
\newblock{ Long-time existence for small amplitude semilinear wave
equations.}
\newblock{ Amer. J. Math.} {\bf 118}(5): 1047--1135 (1996).

\bibitem{L18} Lucente S.,
\newblock{Critical exponents and where to find them.}
{\em Bruno Pini Mathematical Analysis Seminar}  {\bf 9}(1):102--114 (2018).
%




\bibitem{MPbook} Mitidieri E.,  Pohozaev S. I., (2001). 
\newblock{A priori estimates and the absence of solutions of nonlinear partial differential equations and inequalities.} \newblock{\em  Proc. Steklov Inst. Math.} {\bf 234}:1-362. 

\bibitem{NPR16} Nunes do Nascimento W.,  Palmieri A., Reissig  M.,
\newblock{Semi-linear wave models with power non-linearity and scale-invariant time-dependent mass and dissipation.}
\newblock{\em Math. Nachr.}  {\bf 290}(11/12): 1779--1805 (2017). 


\bibitem{OLBC10} Olver F.W.J., Lozier D.W., Boisvert R.F., Clark C.W. (eds.), 
\newblock{\em NIST Handbook of Mathematical Functions.} \newblock{Cambridge University Press, New York} (2010).



\bibitem{Pal18odd}   Palmieri A., 
\newblock{Global existence results for a semilinear wave equation with scale-invariant damping and mass in odd space dimension.} 
\newblock{ In D'Abbicco  M.,  Ebert M.R, Georgiev V., Ozawa T. eds.} \newblock{\em New Tools for Nonlinear PDEs and Application}, \newblock{Trends in Mathematics.} (2019). https://doi.org/10.1007/978-3-030-10937-0$\_$12  


\bibitem{Pal18even} Palmieri A., 
\newblock{A global existence result for a semilinear wave equation with scale-invariant damping and mass in even space dimension.} \newblock{\em  Math. Methods Appl. Sci.} {\bf 42}(8), 2680--2706 (2019). https://doi.org/10.1002/mma.5542.




 

\bibitem{Pal19RF} Palmieri  A., 
\newblock{Integral representation formulae for the solution of a wave equation with time-dependent damping and mass in the scale-invariant case,} 
\newblock{\em  Math. Methods Appl. Sci.} {\bf 44}(17): 13008--13039 (2021).  https://doi.org/10.1002/mma.7603.



\bibitem{Pal20EdSmu} Palmieri  A.,
\newblock{Blow-up results for semilinear damped wave equations in Einstein-de Sitter spacetime.}  
\newblock{\em Z. Angew. Math. Phys.} {\bf 72}, 64 (2021). https://doi.org/10.1007/s00033-021-01494-x

\bibitem{Pal21} Palmieri A.,
\newblock{On the critical exponent for the semilinear Euler-Poisson-Darboux-Tricomi equation with power nonlinearity.}
Preprint, arXiv:2105.09879v2, (2024)


\bibitem{PalRei18} Palmieri  A., Reissig  M., 
\newblock{A competition between Fujita and Strauss type exponents for blow-up of semi-linear wave equations with scale-invariant damping and mass.}
\newblock{\em J. Differential Equations} {\bf 266}(2/3): 1176--1220  (2019). https://doi.org/10.1016/j.jde.2018.07.061.

\bibitem{PalTak22} Palmieri A., Takamura H.,
\newblock{On a semilinear wave equation in anti-de Sitter spacetime: The critical case}
\newblock{\em J. Math. Phys.} {\bf 63}, 111505 (2022). https://doi.org/10.1063/5.0086614

\bibitem{PT18} Palmieri A., Tu Z., 
\newblock{Lifespan of semilinear wave equation with scale invariant dissipation and mass and sub-Strauss power nonlinearity.} \newblock{\em J. Math. Anal. Appl.} {\bf 470}(1): 447--469 (2019). https://doi: 10.1016/j.jmaa.2018.10.015.


\bibitem{Poi} Poisson S.D.,
\newblock{M\'{e}moire sur l’int\'{e}gration des \'{e}quations lin\'{e}aires aux differences partielles.}
\newblock{\em J. de L'\'{E}cole Polytechechnique}, Ser. 1., {\bf 19}:  215--248 (1823).


	\bibitem{Scha85} Schaeffer J.,
\newblock {The equation $u_{tt}-\Delta u = |u|^p$ for the critical value of $p$,}
\newblock { Proc. Roy. Soc. Edinburgh Sect. A.} {\bf 101}(1-2): 31--44 (1985).





\bibitem{Sid84}  Sideris T.C.,
\newblock { Nonexistence of global solutions to semilinear wave equations in high dimensions,}
\newblock { J. Differential Equations} {\bf 52}(3): 378--406 (1984).


\bibitem{Str81} Strauss W.A.,
\newblock{Nonlinear scattering theory at low energy.}
\newblock{\em J. Funct. Anal.} {\bf 41}(1): 110--133  (1981).
	

\bibitem{Sun21} Sun Y.,
\newblock{ Sharp lifespan estimates for subcritical generalized semilinear Tricomi equations.}
\newblock{\em Math. Meth. Appl. Sci.} {\bf 44}(13): 10239--10251 (2021). 

\bibitem{Tak15} Takamura H.,
\newblock{Improved Kato’s lemma on ordinary differential inequality and its application to semilinear wave equations,}
\newblock{Nonlinear Anal.} {\bf 125}: 227--240  (2015).

\bibitem{TakWak11} Takamura H., Wakasa K.,
\newblock{The sharp upper bound of the lifespan of solutions to critical semilinear wave equations in high dimensions,}
\newblock{\em J. Differential Equations} {\bf 251}(4/5):  1157--1171 (2011).

\bibitem{Tri} Tricomi F.,
\newblock{Sulle equazioni lineari alle derivate parziali di $2^\circ$ ordine, di tipo misto.}
\newblock{\em Atti Accad. Naz. Lincei Cl. Sci. Fis. Mat. Natur.} Ser. V {\bf 30}(2): 495--498 (1921).

%
%
%
%
%




\bibitem{TL19}  Tu Z.,  Lin  J., 
\newblock{Lifespan of semilinear generalized Tricomi equation with Strauss type exponent.} 
Preprint, arXiv:1903.11351v2 (2019).

	
	


\bibitem{WakYor19} Wakasa K., Yordanov B.,
\newblock{Blow-up of solutions to critical semilinear wave equations with variable coefficients,}
\newblock{\em J. Differential Equations} {\bf 266}(9): 5360--5376 (2019).

\bibitem{Wak14}  Wakasugi Y., 
\newblock{Critical exponent for the semilinear wave equation with scale invariant damping.}
\newblock{In: Ruzhansky, M.,  Turunen, V. eds.} \newblock{\em Fourier Analysis.} \newblock{Trends in Mathematics. Birkh\"auser, Cham} (2014), https://doi.org/10.1007/978-3-319-02550-6$\_$19.

\bibitem{Wei54} Weinstein A.,
\newblock{The singular solutions and the Cauchy problem for generalized Tricomi equations.}
\newblock{\em  Comm. Pure Appl. Math.} {\bf 7}: 105--116  (1954).

\bibitem{Yag04}  Yagdjian K.,  
\newblock{ A note on the fundamental solution for the Tricomi-type equation in the hyperbolic domain,}
\newblock{\em  J. Differential Equations} {\bf 206}(1): 227--252 (2004).


\bibitem{Yag06}  Yagdjian K.,  
\newblock{ Global Existence for the n-Dimensional Semilinear Tricomi-Type Equations.} \newblock{\em Comm. Partial Differential Equations} {\bf 31}(6): 907--944 (2006).

\bibitem{Yag07} Yagdjian K.,  
\newblock{The self-similar solutions of the Tricomi-type equations,}
\newblock{\em Z. Angew. Math. Phys.} {\bf 58}(4): 612--645 (2007).


\bibitem{Yag10} Yagdjian K.,  
\newblock{ Fundamental solutions for hyperbolic operators with variable coefficients,}
\newblock{\em Rend. Istit. Mat. Univ. Trieste} {\bf 42}, suppl.: 221--243 (2010).


 

\bibitem{Yag15} Yagdjian K.,  
\newblock{Integral transform approach to generalized Tricomi equations.}
\newblock{\em  J. Differential Equations} {\bf 259}(11): 5927--5981 (2015).







\bibitem{YZ06} Yordanov B.T., Zhang Q.S.,
\newblock{Finite time blow up for critical wave equations in high dimensions.}
\newblock{\em J. Funct. Anal.} {\bf 231}(2): 361-374 (2006).


\bibitem{Zhou07}  Zhou Y.,
	\newblock{Blow up of solutions to semilinear wave equations with critical exponent in high dimensions,} 
	\newblock{Chin. Ann. Math. Ser. B} {\bf 28}(2): 205--212 (2007).

\bibitem{ZH14} Zhou Y., Han W., 
\newblock{Life-span of solutions to critical semilinear wave equations.}
\newblock{Commun. Partial. Differ. Equ.} {\bf 39}(3): 439--451 (2014).
	

\end{thebibliography}
\end{document}